\documentclass[11pt,reqno]{amsart}
\usepackage[margin=1.30in]{geometry}
\usepackage{setspace}

\usepackage[utf8]{inputenc}
\usepackage{amsmath,amsfonts,amssymb,amsopn,amscd,amsthm,bbm}
\usepackage{comment}
\usepackage{dsfont}
\usepackage{graphicx}
\usepackage{color}
\usepackage[colorlinks]{hyperref}
\usepackage{epigraph,todonotes}
\usepackage{enumerate}
\allowdisplaybreaks[4]

\DeclareMathOperator{\Hc}{ \textrm{ch} }

\def\1{{\mathbf 1}}

\def\pa{\partial}

\def\m{m}
\def\a{\alpha}

\def\d{\delta}
\def\e{\epsilon}

\def\N{{\mathbb N}}
\def\Z{{\mathbb Z}}

\def\R{{\mathbb R}}

\def\E{{\mathbb E}}

\def\X{{\mathbb X}}

\def\Dc{{\mathcal D}}

\def\Fc{{\mathcal F}}

\def\Hc{{\mathcal H}}

\def\Lc{{\mathcal L}}

\def\Pc{{\mathcal P}}

\def\Sc{{\mathcal S}}

\numberwithin{equation}{section}

\newtheorem{thm}{Theorem}[section]
\newtheorem{proposition}{Proposition}[section]

\newtheorem{definition}{Definition}[section]

\newtheorem{lemma}{Lemma}[section]

\newtheorem{rmk}{Remark}[section]
\newtheorem{assumption}{Assumption}[section]

\newcommand{\ignore}[1]{}
\newcommand{\vertiii}[1]{{\left\vert\kern-0.25ex\left\vert\kern-0.25ex\left\vert #1 
    \right\vert\kern-0.25ex\right\vert\kern-0.25ex\right\vert}}

\title[Comparison of second-order PDEs on Wasserstein space]{Comparison of viscosity solutions for a class of second-order PDEs on the Wasserstein space} 
\author{Erhan Bayraktar}\thanks{E. Bayraktar is partially supported by the National Science Foundation under grant DMS-2106556 and by
the Susan M. Smith chair.}
\address{Department of Mathematics, University of Michigan}
\email{erhan@umich.edu}
\author{Ibrahim Ekren}\thanks{I. Ekren is supported in part by NSF Grant DMS 2007826.}
\address{Department of Mathematics, University of Michigan}
\email{iekren@umich.edu}
\author{Xin Zhang} 
\address{ Department of Finance and Risk Engineering, New York University}
\email{xz1662@nyu.edu}


\begin{document}
\maketitle

\begin{abstract}
We prove a comparison result for viscosity solutions of second-order parabolic partial differential equations in the Wasserstein space. The comparison is valid for semisolutions that are Lipschitz continuous in the measure in a Fourier-Wasserstein metric and uniformly continuous in time. The class of equations we consider is motivated by Mckean-Vlasov control problems with common noise and filtering problems. { The proof of comparison relies on a decomposition of the Wasserstein space, and an application of Ishii's lemma which is tailor-made for the class of equations we consider.} 
\end{abstract}
\section{Introduction}

In this paper, motivated by optimal control of partially observed systems, we show a comparison principle for viscosity solutions of second-order parabolic partial differential equations on Wasserstein space of the form
\begin{align}\label{eq:introparabolic}
-\partial_t v(t,\mu)=h(t, \mu, D_{\mu} v(t,\mu)(\cdot), D_{x \mu} v(t,\mu)(\cdot), \mathcal{H} v(t,\mu)), \quad (t,\mu) \in [0,T) \times \mathcal{P}_2(\mathbb{R}^d),
\end{align}
where we adopt the notion of $L$-derivative w.r.t. the measure variable (see \cite{cdll2019}), and define
\begin{align*}
\mathcal{H}v(t,\mu):= \int D_{x \mu} v(t,\mu)(x) \, d \mu+ \int \int D_{\mu\mu}^2 v(t,\mu)(x,z) \, d \mu d \mu. 
\end{align*}
The equation involves second-order derivatives with respect to the measure, and the Hamiltonian $h$ may be fully nonlinear and degenerate in $\mathcal{H}v$. We actually prove a more general result when $v$ is also dependent on another state variable $y \in \mathbb{R}^d$; see Theorem~\ref{thm:comp}. In filtering problems due to partial observation, one controls the conditional law of the state variable to minimize the payoff, and hence in this framework the value function $v$ depends on the current time and distribution of the unobserved variable. As one application of the comparison principle, we prove that the value function of the stochastic control problem with partial information is the unique viscosity solution of the corresponding Hamilton-Jacobi-Bellman equation; see Theorem~\ref{thm:viscosity_property}. The paper provides the first comparison result for a class of second-order PDEs on a Wasserstein space for Lipschitz sub/super solutions with respect to a Fourier Wasserstein distance under the standard definition of a viscosity solution which is expected to enjoy stability results.

PDEs on Wasserstein space also appear in mean field games and McKean-Vlasov control problems, see e.g. \cite{cdll2019,MR3752669,MR3753660,MR4499277, MR3739204, cosso2021master}, wherein value functions depend on interactions between a large population of particles, and hence the state space becomes $\mathcal{P}_2(\mathbb{R}^d)$, the space of probability measures on $\mathbb{R}^d$ with finite second moment. Using the notion of Wasserstein-$2$ metric $W_2$ from Optimal Transport Theory, this space, which we denote by $(\mathcal{P}_2(\mathbb{R}^d), W_2)$ has an almost-Riemannian geometric structure; see \cite{MR2358290,villani2009optimal}. Various notions of differentiability for functions on Wasserstein space have been defined, and in this paper we adopt the one introduced by Lions in \cite{Lions}. It is stronger than the geometric definition of differentiability, and allows a version of It\^{o}'s formula which is crucial for control problems. For equations on Euclidean space, viscosity theory provides well-posedness for nonlinear degenerate second-order parabolic PDEs, and thus is widely used in stochastic control problems. For second-order equations on Wasserstein space, although it has been shown in \cite{cdll2019} that under uniform elliptic conditions and convexity of coefficients the so-called Master equation in mean field games has regular strong solutions, one still relies on viscosity solutions to tackle nonlinear-degenerate equations.

Let us discuss the main difficulties in proving a comparison result on the Wasserstein space and how we overcame them. Regarding second-order equations on Euclidean space, if both the subsolution $u$ and the supersolution $v$ are regular, the comparison principle can be obtained simply by comparing derivatives of $u$ and $v$ at local maximums of their difference $u-v$. However, solutions of nonlinear degenerate equations are usually not differentiable. Thus, one can adopt the technique of doubling variable, and rely on Ishii's lemma to obtain sub/super jets of $u$ and $v$; see \cite{usersguide}. On Wasserstein space, we also employ the doubling variable technique, and consider $u(\mu)-v(\nu)-\frac{\alpha}{2}\rho_F^2(\mu,\nu)$ where $\mu,\nu\in \mathcal{P}_2(\mathbb{R}^d), \alpha >0$ and $\rho_F$, see \eqref{eq:defrhoF}, is a Fourier-Wasserstein metric to be defined for our purposes. There are two main issues to tackle, $(i)$ due to the lack of local compactness on $\mathcal{P}_2(\mathbb{R}^d)$, $(\mu,\nu) \mapsto u(\mu)-v(\nu)-\frac{\alpha}{2}\rho_F^2(\mu,\nu)$ may not obtain local maximums, $(ii)$ one cannot directly rely on Ishii's lemma, which was so far available in Hilbert spaces. To solve the first one, we apply a smooth variation principle; see \cite{MR4607651,2023arXiv230315160C,cosso2021master}. By adding a smooth perturbation term constructed from a Gaussian regularization of the dyadic transport metric $\rho_{\sigma}$, see \eqref{eq:defrhosigma}, to $u(\mu)-v(\nu)-\frac{\alpha}{2}\rho_F^2(\mu,\nu)$, we obtain local maximums of the perturbed function. After posting this paper, Soner and Yan solved the issue of non-compactness in \cite{2023arXiv230804097M} by an elementary technique of moment penalization, which could be potentially applied to our equation. The second problem is more subtle and is due to the non-smooth and infinite dimensional structure of $\mathcal{P}_2(\R^d)$. 
The key observation in \eqref{eq:introparabolic} is that the second order derivative in measure appears in an integral form, and it is actually of finite-dimensional nature, that is, for any regular function $v: \Pc_2(\R^d) \to \R$.
\begin{align*}
\mathcal{H}v(\mu)= \frac{d^2}{dx^2} v((x+Id)_{\sharp}\mu)|_{x=0} \in \R^{d \times d}. 
\end{align*}
Therefore, similarly to \cite[Section 5]{gangbo2021finite}, we identify $\Pc_2(\R^d)$ with a product space $\R^d \times  \Pc_{2,0}(\R^d)$ via $\mu \mapsto (m(\mu),\Sc_0(\mu))$, where $m(\mu)$ stands for the mean of $\mu$, and $\Sc_0(\mu):=(I_d-m(\mu))_{\sharp} \mu$ is the centered part of $\mu$, and $\Pc_{2,0}(\R^d)$ is the set of probability measures with the mean $0$. Then $\mathcal{H} v$ is the Hessian of $v:\R^d \times\Pc_{2,0}(\R^d)\to \R$ in the first component, and we define $\rho_F^2(\mu,\nu)$ to be the sum of $|m(\mu)-m(\nu)|^2$ and a smooth distance between $\Sc_0(\mu), \Sc_0(\nu)$. Then, in Theorem~\ref{ishii}, we obtain the partial Hessian $\mathcal{H}$ by invoking Ishii's lemma on $\R^d$, and obtain $D_\mu, D_{x \mu}$ derivatives from $\rho_F^2(\mu,\nu)$.

Stochastic filtering problems investigate how to make inferences about an evolving system using partial observations, and we refer to \cite{MR2454694}  for an introduction. Under the assumption of uniform ellipticity, \cite{martini2023kolmogorov} shows the existence of classical solutions to the Kolmogorov equations in the space of measures that arise naturally in filtering problems such as the Zakai and Kushner-Stratonovich equations. For the control problem in the degenerate case, assuming conditional distributions are absolutely continuous with densities in $L^2(\R^d)$, \cite{gozzi2000hamilton} proves the uniqueness of solution invoking the viscosity theory of PDEs on Hilbert space. Viewing $\Pc_2(\R^d)$ as a quotient of Hilbert space, \cite{bandini2019randomized} lifts PDEs on Wasserstein space to the ones on Hilbert space, and provides the uniqueness result. { However, the relation between solutions to the lifted PDEs and the original ones are unclear; see \cite[Remark 3.6]{cosso2021master} for more details. The lift of a smooth function on Wasserstein space may not be second-order Fr\'{e}chet differentiable; see e.g. \cite[Example 2.3]{MR3630288}. Moreover, there is no proper lift version of $D_{x\mu}$ on $L^2$ space. Therefore, in this paper, we would like to study the problem intrinsically on Wasserstein space}. We show a comparison principle for a class of nonlinear degenerate second-order parabolic PDEs on Wasserstein space, and thus give a PDE characterization of the value function in filtering problem. Moreover, \eqref{eq:introparabolic} also covers first-order HJB equations corresponding to McKean-Vlasov control problems.

There is another motivation of this paper. It is well-known that solutions to PDEs can be approximated using machine learning-based numerical methods. On the other hand, some learning problems such as expert prediction can be viewed as finite difference algorithms of PDEs; see \cite{2019arXiv190202368B,2020arXiv200308457B,MR4120922,MR4053484,2020arXiv200813703C,2020arXiv200712732D,JMLR:v24:22-1001, MR4253765}. Therefore, the long-time behavior of such learning problems is characterized by certain parabolic PDEs. In \cite{JMLR:v24:22-1001}, the authors interpreted a learning problem as a discrete time zero-sum game with incomplete information. After properly scaling in both time and space variable, the value function converges to a second-order parabolic equation on Wasserstein space of type \eqref{eq:introparabolic} in a formal way; see \cite[Equation (4.11)]{JMLR:v24:22-1001}. To verify convergence rigorously, viscosity theory has to be established. This paper provides the uniqueness result as a first step, and the the stability of viscosity solutions of \eqref{eq:introparabolic} is an ongoing project.

Let us mention some related references.    The decomposition of $\Pc_{2}(\R^d)$ into $\R^d \times  \Pc_{2,0}(\R^d)$ can be found in \cite{gangbo2021finite}, where the authors define intrinsic viscosity solutions, a notion that is the closest to ours. Compared to \cite{gangbo2021finite} when defining second-order jets, we do not require the existence of derivatives at all directions but only an integral of these derivatives. Our second-order jets being larger, our definition is more stringent than the definition of intrinsic viscosity solutions in \cite{gangbo2021finite}, and any viscosity solution in our sense is an intrinsic viscosity solution in the sense of \cite{gangbo2021finite}. 
Note that \cite{gangbo2021finite} does not provide a comparison of the intrinsic viscosity solution for second-order equations and their main comparison result relies on lifting the functions of measures into functionals on an $L^2$ space and using the comparison of viscosity solutions on a Hilbert space. For the viscosity theory of PDEs on Hilbert space, we refer the readers to \cite{lions1988viscosity1,lions1989viscosity2,lions1989viscosity3,fabbri2017stochastic}. Viscosity solutions of first-order PDEs involving $D_{x\mu}$ have been studied in \cite{MR4604196,2023arXiv230804097M,soner2022viscosity,burzoni2020viscosity,MR4595996, cosso2021master,2023arXiv230815174D}. It is worth noting that our formulation of the Fourier-Wasserstein metric, represented as $\rho_F$, as well as the variational function, symbolized as $\rho_{\sigma}$, drew significant inspiration from the concepts discussed in the first-order case in \cite{soner2022viscosity} and \cite{cosso2021master}, respectively. { We should mention that the same idea of change of variable, $(x,\mu) \mapsto (I_d+x)_{\#} \mu $, and another notion of differentiability were adopted in \cite{2023arXiv230604283B}.  There Bertucci proved a comparison principle for a class of equations where the Hamiltonian appears  in the form of 
\begin{align*}
h(\mu,D_{\mu} v)+\frac{\sigma(t)^2}{2}  \mathcal{H}v(t,\mu)
\end{align*}
for some proper choice of $h$ and deterministic function $\sigma$. 
The proof in this paper relies on reformulating the problem in a Hilbert space and using Ishii's lemma developed for Hilbert spaces in \cite{lions1989viscosity3}.  This change of variable method was also used in the analysis of mean field game with common noise in \cite{cdll2019}, in which case the volatility of common noise is assumed to be a constant. } { We would like to mention that after this paper was posted, the idea of partial Hessian has been adopted also in \cite{2023arXiv231210322C, 2023arXiv231202324D} to prove uniqueness for second-order PDEs.}
As shown in \cite{cox2021controlled}, second-order PDEs on the Wasserstein space are also related to measure-valued martingale optimization problems. 
Relying on the specific martingale structure, that paper manages to reduce the problem to finitely supported measures where the usual viscosity theory can be applied. 
\cite{cox2021controlled} proves a very general uniqueness results under a novel definition of viscosity solution which might not enjoy the stability property of viscosity solutions. PDEs on the Wasserstein space also appear in mean-field optimal stopping problems \cite{MR4604196,MR4613226,2023arXiv230709278P}. 

The remainder of the paper is organized as follows. In Section 2, we will introduce different topological equivalent metrics used in this paper, and provide definitions for differentiability of functions on Wasserstein space. In Section 3, we present a version of Ishii's lemma on Wasserstein space, Theorem~\ref{ishii}, and the comparison result, Theorem~\ref{thm:comp}. Section 4 is devoted to an application of comparison principle in stochastic control with partial observation. Section 5 provides some preliminary estimates for derivatives of $\rho_F$ and $\rho_{\sigma}$. The proofs of Theorem~\ref{ishii}, Theorem~\ref{thm:comp}, and all the arguments in Section 4 can be found in Sections 6, 7, and 8 respectively.

We will end this section by explaining the frequently used notation.

\noindent {\bf{Notation.}} We denote by $\Pc_2(\R^d)$ the set of Borel probability measures $\mu$ in $\R^d$ with a finite second moment, define $m(\mu):=\int x \, \mu(dx) \in \R^d$, and by $\Pc_{2,0}(\R^d)$ the elements of $\Pc_2(\R^d)$ with mean $0$. Both the identity matrix of dimension $d$ and the identity map from $\R^d$ to $\R^d$ are denoted by $I_d$. The centered part of measure $\mu$ is given by $\Sc_0(\mu):=(I_d- m(\mu))_{\sharp} \mu\in \Pc_{2,0}(\R^d)$, and the covariance matrix of $\mu$ is defined as 
\begin{align*}
    V(\mu)=V(\Sc_0(\mu)):=\int (x-m(\mu))(x-m( \mu))^\top  \, \mu(dx) \in \mathcal{S}_d,
\end{align*}
where $\mathcal{S}_d$ stands for the set of positive-semidefinite matrices of dimension $d$. For $M,N \in \mathcal{S}_d$, $|M-N|$ is taken as the Frobenious norm of matrices, i.e., $|M-N|^2:=Tr((M-N)^\top (M-N))$. For any complex number $z$, we denote its conjugate by $z^*$ and its real part by $Re(z)$.

We denote by $\mu*\nu$ the convolution of $\mu, \nu \in \Pc_2(\R^d)$ defined via 
$$\int f(x)(\mu*\nu)(dx)=\iint f(x+y)\mu(dx)\nu(dy) \quad \mbox{ for all }f \in C_b(\R^d).$$
For any $\sigma>0$, $N_{\sigma}$ denotes the $d$-dimensional Gaussian distribution with mean $0$ and covariance matrix $\sigma I_d$. 

For $k \in \R^d$, we define the Fourier basis function $f_k$ by $f_k(x)=(2\pi)^{-d/2}e^{-ik\cdot x}$, and for any $\mu \in \Pc_2(\R^d)$ we take $F_k(\mu)=\int f_k(x) \, \mu(dx)$.   The Fourier transform $\mathcal{F}g$ of a function $g \in L^2(\R^d)$ is defined as $\mathcal{F}g(k):=\int f_k(x)g(x) \, dx.$ Throughout the paper, let us fix $\lambda=d+7$; see Remark~\ref{rmk:choiceoflambda} for a short explanation of this choice. Let $\mathbb{H}_{\lambda}$ be the Hilbert space 
\begin{align*}
    \mathbb{H}_{\lambda}:=\left\{ f \in L^2 (\R^d): \,   (1+|k|^2)^{\frac{\lambda}{2}} \mathcal{F}f  \in L^2(\R^d)\right\},
\end{align*}
where for all $f \in \mathbb{H}_{\lambda}$ we define
\begin{align*}
\lVert f \rVert_{\lambda}^2:= \int (1+|k|^2)^\lambda |\mathcal{F}f(k)|^2 \, dk .
\end{align*}
For any $\mu \in \Pc_2(\R^d)$ and measurable function $f:\R^d \to \R$, we use the short notation $\mu(f)=f([\mu]):=\int f(x) \, \mu(dx)$ if the integral is well-defined. For two probability measures $\mu,\nu\in \Pc_2(\R^d)$, we define $ \lVert \mu-\nu \rVert_{-\lambda}:=\sup_{\lVert f \rVert_{\lambda} \leq 1} (\mu-\nu)(f)$. For any $f:\R^d \to \R$, let us also define its Lipschitz constant and supnorm
$$Lip(f)=\sup_{x,y\in \R^d,x\neq y}\frac{|f(x)-f(y)|}{|x-y|}, \quad  
||f||_\infty:=\sup_{x\in \R^d}|f(x)|.$$



\section{Metrics and derivatives on the space of measures}

\subsection{Different metrics}

We endow $\Pc_2(\R^d)$  with the 2-Wasserstein distance $W_2$, that is, for any $\mu,\nu \in \Pc_2(\R^d)$
\begin{align*}
W_2(\mu,\nu)^2:= \inf_{\pi \in \Pi(\mu,\nu)} \int \frac{1}{2}|x-y|^2 \, \pi(dx,dy),
\end{align*}
where $\Pi(\mu,\nu)$ denotes the collection of probability measures on $\R^d \times \R^d$ with first and second marginals $\mu$ and $\nu$ respectively. 
By \cite[Lemma 2]{auricchio2020equivalence}, we have 
\begin{align}\label{eq:propw}
W_2(\mu,\nu)^2=\frac{1}{2}|m(\mu)-\m(\nu)|^2+W_2(\Sc_0(\mu),\Sc_0(\mu))^2.
\end{align}
As such, $\Pc_2(\R^d)$ admits a natural decomposition as the product of $\R^d$ and $\Pc_{2,0}(\R^d)$ that will be crucial for our comparison result. 

Similarly to \cite{cosso2021master}, we also endow $\Pc_2(\R^d)$ with the metric 
$$W^{(\sigma)}_2(\mu,\nu):=W_2(\mu*N_\sigma,\nu*N_\sigma)$$
so that $(\Pc_2(\R^d),W^{(\sigma)}_2)$ is a complete metric space topologically equivalent to $(\Pc_2(\R^d),W_2)$ thanks to \cite[Lemma 4.2]{cosso2021master}. Fix $\sigma>0$ and let $B_0=(-1,1]^d$. For $l\geq 0$, $\mathfrak{P}_l$ denotes the partition of $B_0$ into $2^{dl}$ translations of $(-2^{-l},2^{-l}]^d$ and for every $n\geq 1$, $B_n=(-2^n,2^n]^d\backslash(-2^{n-1},2^{n-1}]^d$.
With $\delta_{n,l}:=2^{-(4n+2dl)}$, we define 
\begin{align}\label{eq:defrhosigma}
 \rho^2_\sigma(\mu,\nu)&:=|m( \mu)-m(\nu)|^2\\
 &+\sum_{n\geq 0}2^{2n}\sum_{l\geq 0}2^{-2l}\sum_{B\in \mathfrak{P}_l}\sqrt{|((\Sc_0(\mu)-\Sc_0(\nu))*\mathcal{N}_\sigma)((2^nB)\cap B_n)|^2+\delta_{n,l}^2}-\delta_{n,l}  \notag
\end{align}
which is gauge type function for the metric space $\left(\Pc_2(\R^d), W^{(\sigma)}_2\right)$ due to \cite[Lemma 4.3]{cosso2021master} and \eqref{eq:propw}.

Besides $W^{(\sigma)}_2$, we define another metric $\rho_F: \Pc_2(\R^d) \times \Pc_2(\R^d) \to [0,\infty)$ via 
\begin{align}\label{eq:defrhoF}
    \rho^2_F(\mu,\nu):=|m( \mu)-m( \nu)|^2+|V(\mu)-V(\nu)|^2 + \int_{\R^d} \frac{|F_k(\Sc_0(\mu))-F_k(\Sc_0(\nu))|^2}{(1+|k|^2)^\lambda} \, dk.
\end{align}
Note that $\rho$ is well defined on $\Pc_2(\R^d) \times \Pc_2(\R^d)$ for any $\lambda > \frac{d}{2}$. Recall that we take $\lambda=d+7$.  According to the Sobolev embedding theorem \cite[Corollary 9.13]{MR2759829},  given that $\lambda-\frac{d}{2}=\frac{d}{2}+7>1$, the Lipschitz constant and supnorm are bounded by the $\lambda$-Sobolev norm, i.e., 
$$Lip(f)+\lVert f \rVert_\infty\leq C \lVert f \rVert_{\lambda}.$$

{
\begin{rmk}\label{rmk:choiceoflambda}
For the purpose of the Sobolev embedding and the proof of the comparison principle Theorem~\ref{thm:comp}, it is sufficient to take $\lambda>d/2+2$. However, to verify the Lipschitz property of the value function in Section~\ref{ss.controlproblem}, we need larger $\lambda$ and please see \eqref{eq:lambdagood} for a detailed computation. 
\end{rmk}}

Moreover, for any $\mu \in \Pc_2(\R^d), f \in \mathbb{H}_{\lambda}$ and finite signed measure $\eta$ on $\R^d$, due to Fourier inversion formula, it can be easily seen that 
\begin{align*}
    \eta(f)=\eta(\mathcal{F}^{-1} \mathcal{F}f)&= \frac{1}{(2\pi)^d}\int  \, \eta(dy) \int e^{i y \cdot k} \mathcal{F}f(k) \, d k=  \frac{1}{(2\pi)^d}\int F_{-k}(\eta) \, \mathcal{F}f(k) \, dk \\
    &\leq  \frac{1}{(2\pi)^d}\lVert f \lVert_{\lambda} \sqrt{\int \frac{|F_k(\eta))|^2}{(1+|k|^2)^{\lambda}} \, dk }, 
\end{align*}
and $$\sup_{\lVert f \rVert_{\lambda} \leq 1} \eta(f) \leq \frac{1}{(2\pi)^d}\sqrt{\int \frac{|F_k(\eta)|^2}{(1+|k|^2)^{\lambda}} \, dk }.$$
This observation leads to the following lower bound on $\rho_F$. 
\begin{lemma}\label{lem:metricrho}
There is a constant $C$ such that for any $\mu,\nu \in \Pc_2(\R^d)$
\begin{align*}
  \lVert \mu-\nu \rVert_{-\lambda} := \sup_{ \lVert f \rVert_{\lambda} \leq 1} \int f \, d(\mu-\nu) \leq C \rho_F(\mu,\nu). 
\end{align*}
\end{lemma}
\begin{proof}
Taking $\tilde \mu= (I_d+m(\nu-\mu))_{\sharp}\mu$, for any $f \in \mathbb{H}_{\lambda}$ we have that
\begin{align*}
\int f \, d(\mu-\nu)=&\int f(x) \, (\mu-\tilde\mu)(dx) + \int f(x) \, (\tilde \mu-\nu)(dx) \\
=& \int f(x)-f(x+m(\nu-\mu)) \, \mu(dx) + \int f(x) \, (\tilde \mu-\nu)(dx) \\ 
\leq & |m( \mu )-m (\nu)| Lip(f) +  \frac{1}{(2\pi)^d}\lVert f \rVert_{\lambda}\sqrt{\int \frac{|F_{k}(\tilde \mu)-F_{k}(\nu)|^2}{(1+|k|^2)^{\lambda}} \, dk } \leq C \rho_F(\mu,\nu)
\end{align*}
where in the last inequality we use the fact that thanks to $m(\tilde \mu)=m(\nu)$ one has
$$
|F_k(\tilde \mu)-F_k(\nu)|^2=|F_k(\Sc_0(\mu))-F_k(\Sc_0(\nu))|^2.
$$
\end{proof}

Before stating the next lemma, we recall that the completeness of a metric space is a property of its metric and not its topology. 
\begin{lemma}\label{lem:equivalence}
     $(\Pc_2(\R^d), \rho_F)$ is a {\it non-complete}  metric space which is topologically equivalent to $(\Pc_2(\R^d), W_2)$ and $(\Pc_2(\R^d),W_2^{(\sigma)})$.
\end{lemma}

\begin{proof}
Note that for $\mu_n,\mu\in \Pc_2(\R^d)$, the convergence $\rho_F(\mu_n,\mu)\to 0$
means the convergence of the first two moments of the measures and the pointwise convergence of characteristic functions. Thanks to \cite[Theorem 6.9]{villani2009optimal} and \cite[Theorem 6.3]{MR4226142}, this implies distributional convergence and also that $W_2(\mu_n,\mu)\to 0$. Together with \cite[Lemma 2]{auricchio2020equivalence}, we conclude that $\rho_F$ and $W_2$ are topologically equivalent. Due to
\cite[Lemma 4.3]{cosso2021master}, the topologies generated by $W_2$ and $W^{(\sigma)}_2$ are equal as well. 

The following example shows the lack of completeness of $\rho_F$ which also implies that the metrics are not metrically equivalent. Let us take $$\mu_n=\left(1-\frac{1}{n} \right)\delta_0+\frac{1}{2n}\delta_{-\sqrt{n}e_1}+\frac{1}{2n}\delta_{\sqrt{n}e_1}, \quad n \in \mathbb{N},$$ 
where $e_1=(1,0,0,\dotso,0) \in \R^d$. For each $n$, $\mu_n$ has mean $0$ and variance $e_1^\top e_1$. Due to Sobolev embedding, for any $n,l \geq 1$
\begin{align*}
    \rho_F(\mu_n,\mu_m)& = \sqrt{\int \frac{|F_{k}(\mu_n)-F_{k}(\mu_m)|^2}{(1+|k|^2)^{\lambda}} \, dk }= \sup_{\lVert f \rVert_{\lambda} \leq 1} \int f \, d(\mu_n-\mu_m)  \\
    & \leq C \sup_{ Lip(f) \leq 1}\int f \, d(\mu_n-\mu_m) \leq W_1(\mu_n,\mu_m).
\end{align*}
Since $W_1(\mu_n,\delta_0)\to 0$ as $n\to \infty$, $(\mu_n)_{n\geq 1}$ is a Cauchy sequence w.r.t. $W_1$ and also Cauchy w.r.t. $\rho_F$. However, $\rho_F(\mu_n,\delta_0) \geq 1$ due to the variance part, and hence there is no limit point w.r.t. $\rho_F$. Therefore $\rho_F$ is not complete. 

\end{proof}

We have introduced two functions on $(\Pc_2(\R^d))^2$. $\rho_F$ is a non-complete metric associated the Fourier-Wasserstein metrics in \cite{auricchio2020equivalence,carrillo2007contractive,toscani1999probability}. As observed in \cite{soner2022viscosity}, for first-order equations on compact domains, this function enjoys good differentiability and symmetry properties and will be used for the classical doubling of variables methods of viscosity solutions. Compared to \cite{soner2022viscosity}, the mean term $m$ in \eqref{eq:defrhoF} allows us to easily compute the partial Hessian of $\rho_F$. The presence of the $V$ term allows us to claim the topological equivalence of $\rho_F$ to $W_2$. However, this $V$ term is also the precise reason why completeness fails (owing to the unboundedness of the domain). Indeed, in the proof of Lemma \ref{lem:equivalence}, we are able to obtain the convergence $V(\mu_m)$. However, the limit of $V(\mu_m)$ is not $V(\delta_0)$ and the space $(\Pc_2(\R^d), \rho_F)$ lacks completeness. If the domain were compact, as noted by \cite{soner2022viscosity}, these metrics are equivalent and $(\Pc_2(\R^d),\rho_F)$ is complete.

$\rho_\sigma$ is a modification of the gauge-type function of \cite{cosso2021master} allowing us to use the Borwein-Preiss variational principle in \cite{MR2144010}. Note that the two metrics $\rho_F, W^{(\sigma)}_2$ are topologically equivalent to $W_2$. Although these metrics lead to different sets of uniformly continuous functions, they have the same set of continuous and semicontinuous functions.

\subsection{First-order derivatives}
For a function $f:\R^d\mapsto \R$ we define the quadratic weighted norm 
\begin{align}\label{eq:quad}
||f||_{\mathfrak{q}}:=\sup_{x\in \R^d}\frac{|f(x)|}{1+|x|^2}    
\end{align}
and denote by $B_{\mathfrak{q}}$ the set of Borel measurable functions $f$ with at most quadratic growth, i.e. $||f||_{\mathfrak{q}}<\infty$.

Following the definition of \cite{cdll2019}, a function $u : \Pc_2(\R^d)\mapsto \R$ is said to be linearly differentiable if 
there exists a continuous function\footnote{$(\Pc_2(\R^d), \rho_F)$ and $(\Pc_2(\R^d), W_2)$ being topologically equivalent, we do not need to specify the underlying metric on $\Pc_2(\R^d)$.}
$$\frac{\d u}{\d \mu}:\Pc_2(\R^d)\times \R^d\mapsto \R$$ 
so that  $$\left\|\frac{\d u}{\d \mu}(\mu,\cdot)\right\|_{\mathfrak{q}}<\infty$$
and for all $\mu,\mu'\in \Pc_2(\R^d)$
we have that
\begin{align*}
\lim_{h\to 0}\frac{u(\mu +h (\mu'-\mu))-u(\mu )}{h}=\int \frac{\d u}{\d \mu}(\mu ,x)d(\mu '-\mu)(x).
\end{align*}
In case that $x\mapsto \frac{\d u}{\d \mu}(\mu,x)$ is differentiable, we say $u$ is $C^1$ and define 
\begin{align*}
D_\mu u(\mu ,x)=D_x\frac{\d u}{\d \mu}(\mu ,x)\in \R^d.
\end{align*}
If $\left\|D_\mu u (\mu,\cdot)\right\|_{\mathfrak{q}}<\infty$, for notational simplicity we denote for all $\mu'\in \Pc_2(\R^d)$
\begin{align*}
D_\mu u(\mu ,[\mu'])&=\int D_\mu u(\mu ,x)d\mu'(x)\in \R^d.
\end{align*}
It is shown in \cite[Proposition 2.3]{cdll2019} that $D_\mu u$ can be understood as a derivative $u$ along push-forward directions, meaning for all Borel measurable bounded vector field $\phi:\R^d\mapsto \R^d$,
we have
$$\lim_{h\to 0}\frac{u((I_d+h \phi)_\sharp \mu)-u(\mu )}{h}=\int\phi^\top(x) D_\mu u(\mu ,x) \,d\mu(x).$$
In particular for $\phi(x)=v\in \R^d$, we have that 
$$\lim_{h\to 0}\frac{u((I_d+h v)_\sharp \mu)-u(\mu )}{h}=v^\top D_\mu u(\mu ,[\mu]).$$

\subsection{Second-order derivatives}
\subsubsection{Full $C^2$-regularity}
Similarly to \cite{cdll2019,chow2019partial}, if $D_\mu u(\mu ,x)$ is sufficiently smooth, we also define $D_{x \mu} u(\mu ,x)=D_xD_\mu u(\mu ,x)=D^2_{xx}\frac{\d u}{\d \mu}(\mu ,x)\in \Sc^d$ and $D^2_{\mu\mu}u(\mu ,x,z)=D_\mu(D_\mu u(\mu ,x))(z)$. If $\left\|D_{x \mu} u (\mu,\cdot)\right\|_{\mathfrak{q}}<\infty$ and \begin{align}\label{eq:quad2}
    \left\|D^2_{\mu\mu}u(\mu ,\cdot)\right\|_{\mathfrak{q}}:=\sup\frac{|D^2_{\mu\mu}u(\mu ,x,z)|}{1+|x|^2+|z|^2}<\infty,
\end{align} we define the partial Hessian
\begin{align}
\Hc u(\mu )= D^2_{\mu\mu}u(\mu ,[\mu],[\mu])+D_{x\mu} u(\mu ,[\mu])\in \Sc_d\label{eq:defhess}.
\end{align}
For a function $U:\R\times \Pc_2(\R)\mapsto \R$ we also define the cross derivative 
$$ D^2_{y\mu}U(y,\mu,x):= D_y D_{\mu}U(y,\mu,x)$$
and the partial cross derivative
\begin{align}\label{eq:defcrossder}
    \Dc^2_{y\mu}U(y,\mu):=\int D^2_{y\mu}U(y,\mu,x)\,\mu(dx)=D^2_{y\mu}U(y,\mu,[\mu])\in \R^{d\times d}.
\end{align}

For fixed terminal time $T>0$, we denote $\Theta:= [0,T]\times \R^d\times \Pc_2 (\R^d)$, $\mathring{ \Theta}:= [0,T)\times \R^d\times \Pc_2 (\R^d)$, and $\theta=(t,y,\mu)\in \Theta$ the generic element of $\Theta$. 
We define the following metric and gauge-type function on $\Theta$ by
\begin{align}
    d^2_F(\theta,\tilde \theta):=d^2_F((t,y,\mu),(\tilde t,\tilde y,\tilde \mu)):=|t-\tilde t|^2+|y-\tilde y|^2+\rho_F^2(\mu,\tilde \mu)\label{eq:defdF}\\
    d^2_\sigma(\theta,\tilde \theta):=d^2_\sigma((t,y,\mu),(\tilde t,\tilde y,\tilde \mu)):=|t-\tilde t|^2+|y-\tilde y|^2+\rho^2_\sigma(\mu,\tilde \mu)\label{eq:defdsigma}
\end{align}
for $(\theta,\tilde \theta)=((t,y,\mu),(\tilde t,\tilde y,\tilde \mu))\in \Theta^2$.

Our goal is to provide assumptions on the function 
$$h:\Theta \times \R\times \R^d\times \Sc_d\times B^d_{\mathfrak{q}}\times B^{d\times d}_{\mathfrak{q}}\times \R^{d\times d}\times \Sc^d\to \R $$ 
to have comparison of Lipschitz continuous viscosity solutions of
\begin{align}\label{eq:pde}
        -\pa_t U(\theta)-h(\theta,U(\theta),D_y U(\theta),D_{yy}^2 U(\theta),D_\mu U(\theta),D_{x\mu}U(\theta),\Dc^2_{y\mu}U(\theta) ,\Hc U(\theta))=0,
\end{align}
for $\theta=(t,y,\mu)\in \mathring{\Theta}$.
In equation \eqref{eq:pde}, the subscript $y$ is the derivative with respect to the state variable $y$, 
$(D_\mu U(\theta),D_{x\mu}U(\theta) )=(D_\mu U(\theta,x),D_{x\mu}U(\theta,x) )_{x\in \R^d}$ are functions of the dummy variable $x$, and 
$(D_y U(\theta),D_{yy}^2 U(\theta),\Dc^2_{y\mu}U(\theta),\Hc U(\theta))\in \R^d\times \Sc_d\times \R^{d\times d}\times \Sc_d$
are finite dimensional. 

We now define the concept of full $C^2$-regular functions and full $C^2$-regular solutions of \eqref{eq:pde}.
\begin{definition}
    We say that a function $U:\Theta\mapsto \R$ is full $C^2$-regular if 
    $$(\pa_t U,D_y U,D_{yy}^2 U,D_\mu U,D_{x\mu} U,D^2_{y\mu} U,D^2_{\mu\mu} U)\in \R\times \R^d\times \Sc_d\times B^d_{\mathfrak{q}}\times B^{d\times d}_{\mathfrak{q}}\times B^{d\times d}_{\mathfrak{q}}$$ is continuous in $(t,x,z,y,\mu)$. 
    
    $U$ is a full $C^2$-regular solution (resp. subsolution, supersolution) of \eqref{eq:pde}, if it is a full $C^2$-regular function and 
    \begin{align*}
        &-\pa_t U(\theta)-h(\theta,U(\theta),D_y U(\theta),D_{yy}^2 U(\theta),D_\mu U(\theta),D_{x\mu}U(\theta),\Dc^2_{y\mu}U(\theta) ,\Hc U(\theta))\\
        &=0 \, (\mbox{resp. }\leq 0,\, \geq 0 ),
\end{align*}
for all $\theta\in \mathring{\Theta}$ where $\Dc^2_{y\mu}U(\theta)$ is defined by \eqref{eq:defcrossder} and $\Hc U(\theta)$ is defined by \eqref{eq:defhess}.
\end{definition}

\subsubsection{Partial $C^2$-regularity}
In \eqref{eq:pde}, we do not need $D^2_{\mu\mu}U(\theta ,x,z),D_y D_{\mu}U(\theta,x)$ for all $(x,z)$ but an integral of these functions. Inspired by \cite{chow2019partial,gangbo2021finite}, we aim to define the concept of partial $C^2$-regularity which is luckily the regularity of the test functions we are able to construct with the doubling of variable methodology in our proof of comparison of viscosity solutions. 
Our definition of partial regularity is motivated by the following lemma, which gives an intrinsic definition of $\Hc U,\Dc^2_{y\mu}U$ without relying on the existence of $D^2_{\mu\mu}U(\theta ,x,z)$ and $D_y D_{\mu}U(\theta,x)$.
\begin{lemma}\label{lem:hespartial}
    Let $U$ be a full $C^2$-regular function. 
    Then 
    \begin{align*}
       \frac{\pa^2}{\pa z^2}U(t,y,(I_d+z)_\sharp\mu)|_{z=0}&= \Hc U(t,y,\mu)\\
       \frac{\pa^2}{\pa z\pa y}U(t,y,(I_d+z)_\sharp\mu)|_{z=0}&= \Dc^2_{y\mu} U(t,y,\mu).
    \end{align*}
\end{lemma}
\begin{proof}
    The result is a consequence of \cite[Proposition 3.4]{chow2019partial}.
\end{proof}
Now we define the concept of partial $C^2$-regularity and the solution property of partially $C^2$-regular functions. 
\begin{definition}
    We say that a function $U:\Theta\mapsto \R$ is partially $C^2$-regular if the function $(t,y,z,\mu )\in \R^d\times\Theta\mapsto U(t,y,(I_d+z)_\sharp\mu)$ is $C^{1,2,2}$ in $(t,y,z)$ and continuous in all its variables, and $(D_\mu U,D_{x\mu} U)\in B^d_{\mathfrak{q}}\times B^{d\times d}_{\mathfrak{q}}\times B^{d\times d}_{\mathfrak{q}}$ exist and they are continuous in $(t,x,y,\mu)$. 

    For a partially $C^2$-regular function, we define $\Hc U$ and $\Dc^2_{y\mu} $ as 
    \begin{align*}
        \Hc U(\theta)&:= \frac{\pa^2}{\pa z^2}U(t,y,(I_d+z)_\sharp\mu)|_{z=0}, \\
        \Dc^2_{y\mu} U(\theta)&:= \frac{\pa^2}{\pa z\pa y}U(t,y,(I_d+z)_\sharp\mu)|_{z=0}. 
    \end{align*}
    
    We say that a partially $C^2$-regular function $U:\Theta\mapsto \R$ is a partially $C^2$-regular solution (resp. subsolution, supersolution) of \eqref{eq:pde}, if 
    $$(\pa_t U,D_y U,D_{yy}^2 U,D_\mu U,D_{x\mu}U,\Dc^2_{y\mu}U ,\Hc U) \in \R\times \R^d\times \Sc_d\times B^d_{\mathfrak{q}}\times B^{d\times d}_{\mathfrak{q}}\times \R^{d\times d}\times \Sc_d$$
    is continuous in all of its variables and 
    \begin{align*}
        &-\pa_t U(\theta)-h(\theta,U(\theta),D_y U(\theta),D_{yy}^2 U(\theta),D_\mu U(\theta),D_{x\mu}U(\theta),\Dc^2_{y \mu}U(\theta) ,\Hc U(\theta))\\
        &=0\,(\mbox{resp. }\leq 0,\, \geq 0 ),
\end{align*}
for all $\theta\in \mathring{\Theta}.$
\end{definition}
\begin{rmk}
\begin{itemize}
    \item[(i)] Thanks to \eqref{eq:defdF}, for fixed $\tilde \mu\in \Pc_2(\R^d)$, the mapping
    $$\mu\in \Pc_2(\R^d)\mapsto  d^2_F(\mu,\tilde \mu)$$ is partially $C^2$-regular with 
    $$ \Hc d^2_F(\mu,\tilde \mu)=2I_d\mbox{ and }\Dc^2_{y\mu}d^2_F((t,y,\mu),\tilde \theta)=0.$$
    
    \item[(ii)] According to \cite[Lemma 2]{auricchio2020equivalence}, $W_2(\mu,\nu)^2=W_2(\mathcal{S}_0(\mu), \mathcal{S}_0(\nu))^2 +|m(\mu)-m(\nu)|^2 $, and hence $\mu\in \Pc_2(\R^d)\mapsto  W_2(\mu,\nu)^2$ admits a partial Hessian $\Hc$. However, it does not admit a first-order derivative at some points of $ \Pc_2(\R^d)$. As such, it is not partially $C^2$-regular.
    
    \item[(iii)] Given that we have relaxed the definition of smoothness, a natural question is what class of stochastic dynamics on the Wasserstein space is, for which an Ito formula still holds. We provide an answer to this question in Section \ref{ss.controlproblem}.
\end{itemize}
\end{rmk}

\section{Notion of Viscosity solutions and our main results}
The objective of this section is to use the fact that \eqref{eq:pde} only requires the partial Hessian and partial cross derivative in second-order terms to define the (partial) second-order jets. This leads to a novel definition of the viscosity property and an Ishii lemma. As an application of this result, we provide the promised comparison result.
\subsection{Second-order jets}
\begin{definition}
    
    Let $U:\Theta \mapsto \R$ be a locally bounded function and $\theta \in \mathring{\Theta}$. We define the (partial) second-order superjet $J^{2,+} U(\theta)\subset  \R\times \R^d\times \Sc_d\times B^d_{\mathfrak{q}}\times B^{d\times d}_{\mathfrak{q}}\times \R^{d\times d}\times \Sc_d$ of $U$ at $\theta$ by
    \begin{align*}
        J^{2,+} U(\theta):=\Big\{&(\pa_t \phi(\theta),D_y \phi(\theta),D_{yy}^2 \phi(\theta),D_\mu \phi(\theta),D_{x\mu}\phi(\theta),\Dc_{y\mu}^2\phi(\theta) ,\Hc \phi(\theta)):\\
        &U-\phi\mbox { has a local maximum at }\theta,\, \phi \mbox{ is partially }C^2\mbox{-regular}\Big\}\notag
    \end{align*}
    and subjet $J^{2,-} U(\theta)=-J^{2,+}(- U)(\theta)$.
    Recalling that the topology of sets $B^d_{\mathfrak{q}}$ and $B^{d \times d}_{\mathfrak{q}}$ are induced by \eqref{eq:quad} and \eqref{eq:quad2} respectively, we define the closure of the jets as
    \begin{align*}
        \notag\bar J^{2,+} U(\theta)=&\Big\{(b,p,X^{11},f,g,X^{12},X^{22})\in \R\times \R^d\times \Sc_d\times B^d_{\mathfrak{q}}\times B^{d\times d}_{\mathfrak{q}}\times \R^{d\times d}\times \Sc_d:\\
        \notag&U-\phi^n\mbox { has a local maximum at }\theta^n,\, \phi^n \mbox{ is partially }C^2\mbox{-regular}\\
        &(\pa_t \phi^n(\theta^n),D_y \phi^n(\theta^n),D_{yy}^2 \phi^n(\theta^n),D_\mu \phi^n(\theta^n),D_{x\mu}\phi^n(\theta^n),\Dc^2_{y\mu}\phi^n(\theta^n) ,\Hc \phi^n(\theta^n))\notag\\
        &\to(b,p,X^{11},f,g,X^{12},X^{22}),\,\theta^n\to \theta,\,U(\theta^n)\to U(\theta) \Big\},
    \end{align*}
    and also $\bar J^{2,-} U(\theta)=-\bar J^{2,+} (-U)(\theta)$.
\end{definition}

\begin{definition}
    We say that a locally bounded function $U:\Theta\mapsto \R$ is a viscosity subsolution (resp. supersolution) to \eqref{eq:pde} if for all $\theta\in \mathring{\Theta}$ and $(b,p,X^{11},f,g,X^{12},X^{22})\in  J^{2,+} U(\theta)$ (resp. $\in J^{2,-} U(\theta)$), one has the inequality
    \begin{align}\label{eq:viscodef}
        &-b-h(\theta,U(\theta),p,X^{11},f,g,X^{12} ,X^{22})\leq 0\,(\mbox{resp.} \geq 0 ).
\end{align}
\end{definition}
\begin{rmk}
    In our definition of viscosity solution, the test functions are allowed to have partial regularity rather than full regularity as it is the case in \cite[Section 5]{gangbo2021finite}. Thus, compared to a definition with full regularity, with our definition one needs to check the inequality \eqref{eq:viscodef} with more test functions $\phi$ as the jets are larger. As noted in \cite[Remark 3.12]{ekren2016viscosity}, our more stringent definition of viscosity property (with a larger set of test functions) is expected to simplify the proof of the comparison of viscosity solutions. 
\end{rmk}

\subsection{Second order jets on the Wasserstein space}
The main technical contribution of the paper is the following version of Ishii's Lemma, whose proof is given in Section \ref{s.proofishii}. The result can be seen as an infinite-dimensional extension of \cite[Theorem 3.2 and Theorem 8.3]{usersguide} or \cite[Proposition II.3]{ishii1990viscosity}. Unlike \cite{fabbri2017stochastic, lions1989viscosity3}, we cannot rely on any underlying Hilbert space structure. The finite dimensionality of the partial Hessian $\Hc u(\theta)$ is crucial in the proof of the Theorem. 
\begin{thm}\label{ishii}
    Assume that $u$, $-v:\Theta\mapsto \R$ are upper semicontinuous and that the function
    \begin{align}\label{cond:unifc0}
        (t,y,z)\in [0,T]\times (\R^d)^2\mapsto (u(t,y,(I_d+z)_\sharp \mu),v(t,y,(I_d+z)_\sharp \mu))
    \end{align}
    is continuous for fixed $\mu\in \Pc_2(\R^d)$.
    Let $\a>0$ and $((t^*,y^*,\mu^*),(\tilde t^*,\tilde y^*,\tilde \mu^*))=(\theta^*,\tilde \theta^*)\in \mathring{\Theta}^2$ be a strict global maximum of 
    $$(\theta,\tilde \theta)\in \Theta^2\mapsto u(\theta)-v(\tilde \theta)-\frac{\a}{2} d_F^2(\theta,\tilde \theta).$$
    Denote
    \begin{align}
        \Lc(\mu,\eta,\nu):=&2\int \frac{Re(F_k(\Sc_0(\mu))(F_k(\Sc_0(\eta))-F_k(\Sc_0(\nu)))^*)}{(1+|k|^2)^\lambda}dk\\
        &+2Tr(V^\top(\Sc_0(\mu))(V(\Sc_0(\eta))-V(\Sc_0(\nu)))),\notag\\
        \Psi(\mu):=&2\rho^2_F(\Sc_0(\mu),\Sc_0(\mu^*))+\Lc(\Sc_0(\mu),\Sc_0(\mu^*),\Sc_0(\tilde \mu^*))\label{eq:defpsi},\\
        \tilde \Psi(\mu):=& 2\rho^2_F(\Sc_0(\mu),\Sc_0(\tilde \mu^*))-\Lc(\Sc_0(\mu),\Sc_0(\mu^*),\Sc_0(\tilde \mu^*)).\notag
    \end{align}

    Then for any $\e>0$, there exist $X=\begin{pmatrix}
    X^{11}&X^{12}\\X^{21}&X^{22}
\end{pmatrix}
,\tilde X=\begin{pmatrix}
    \tilde X^{11}&\tilde X^{12}\\\tilde X^{21}&\tilde X^{22}
\end{pmatrix}
\in \Sc_{2d}$ so that 
\begin{align*}
    &\left(\a(t^*-\tilde t^*),\a(y^*-\tilde y^*),X^{11},\a(m(\mu^*)-m({\tilde \mu}^*))+\a\frac{D_\mu\Psi(\mu^*)}{2},\a\frac{D_{x \mu}\Psi(\mu^*)}{2},X^{12},X^{22}\right)\\
    &\in \bar J^{2,+} u(\theta^*),\\
    &\left(\a(t^*-\tilde t^*),\a (y^*-\tilde y^*),-\tilde X^{11},\a (m(\mu^*)-m( {\tilde \mu}^*))-\a\frac{D_\mu\tilde \Psi(\tilde \mu^*)}{2},-\a\frac{D_{x\mu}\tilde \Psi(\tilde \mu^*)}{2},-\tilde X^{12} ,-\tilde X^{22}\right)\\
    &\in \bar J^{2,-} v(\theta^*),\mbox{ and}\\
    &-\left(\frac{1}{\e}+2\a\right)I_{4d}\leq 
\begin{pmatrix}
    X&0\\0&\tilde X
\end{pmatrix}
\leq 
(\a+2\e\a^2)
\begin{pmatrix}
    I_{2d}&-I_{2d}\\-I_{2d}&I_{2d}
\end{pmatrix}.
\end{align*}
\end{thm}

\subsection{Comparison of viscosity solutions} 
Before stating our main theorem, we list the necessary assumptions. 
\begin{assumption}\label{assumeh}
    \begin{itemize}
    \item[(i)] Assume that the Hamiltonian $h$ is continuous in all variables and satisfies the following Lipschitz continuity assumption 
\begin{align}
    &|h(\theta,u,p,X^{11} ,f,g,X^{12} ,X^{22})-h( \theta,\tilde u,\tilde p, \tilde X^{11},\tilde f ,\tilde g, X^{12} , \tilde X^{22})| \label{eq:unifh1}\\
    &\notag\leq L_h|u-\tilde u|+L_h \left(1+|y|+\int |x|^2 \mu(dx)\right)  \\
    &\quad \times\left(||f-\tilde f||_{\mathfrak{q}}+||g-\tilde g||_{\mathfrak{q}}+|p-\tilde p|+|X^{11}-\tilde X^{11}|+|X^{22}-\tilde X^{22}|\right)\notag
\end{align}
for some constant $L_h$.
\item[(ii)] For any $C>0$, there exists a modulus of continuity $\omega_h$ that may depend on $C$ so that for all $\e>0$, $X,\tilde X\in \Sc_{2d}$ satisfying 
\begin{align}\label{eq:condcomp}
-\frac{3}{\e}I_{4d}\leq 
\begin{pmatrix}
    X&0\\0&\tilde X
\end{pmatrix}
\leq 
\frac{3}{\e}
\begin{pmatrix}
    I_{2d}&-I_{2d}\\-I_{2d}&I_{2d}
\end{pmatrix},
\end{align}
$(\theta^*,\tilde \theta^*)=((t^*,y^*,\mu^*),(\tilde t^*,\tilde y^*,\tilde \mu^*)) \in \mathring{\Theta}^2$ with $d_F((t^*,y^*,\mu^*),( t^*,\tilde y^*,\tilde \mu^*))\leq C\e$, and $r\in\R$ we have the inequality 
 \begin{align}
     &h\left(\theta^*,r, \frac{y^*-\tilde y^*}{\e},( X)^{11}, \frac{m(\mu^*-\tilde \mu^*)}{\e}+\frac{D_\mu\Psi(\mu^*)}{2\e},\frac{D_{x\mu}\Psi(\mu^*)}{2\e},( X)^{12},( X)^{22}\right)\notag\\
     &\leq h\left(\tilde \theta^*,r, \frac{y^*-\tilde y^*}{\e},-(\tilde X)^{11} ,\frac{m(\mu^*-\tilde \mu^*)}{\e}+\frac{D_\mu\Psi(\mu^*)}{2\e},\frac{D_{x\mu}\Psi(\mu^*)}{2\e},-(\tilde X)^{12} ,-(\tilde X)^{22}\right)\notag\\
     & \quad\quad+\omega_h\left(d_F(\theta^*,\tilde \theta^*)\right)\left(1+|y^*|+\int |x|^2\mu^*(dx)\right)\label{eq:contpsi}
\end{align}
where $\Psi$ is defined by \eqref{eq:defpsi} and its derivatives computed in \eqref{eq:derpsi}.

\item[(iii)] For all $(\theta,u,p,f,g)\in \Theta\times \R\times \R^d\times B^{d}_{\mathfrak{q}}\times B^{d\times d}_{\mathfrak{q}}$ the mapping 
    \begin{align}\label{eq:mono}
        \begin{pmatrix}
       X^{11}&X^{12}\\X^{12}&X^{22} 
    \end{pmatrix}\in \Sc_{2d}\mapsto h
    (\theta,u,p,X^{11} ,f,g,X^{12} ,X^{22})
    \end{align}
    is increasing for the order of symmetric matrices.
\end{itemize}

\end{assumption}

{
\begin{rmk}
In the next section, for problems of stochastic control with partial observation, we provide a simple sufficient condition for Assumption~\ref{assumeh}. In this example, to verify the regularity in the measure component, we rely on the fact the Hamiltonian is actually a minimization over integrals of regular coefficients w.r.t. the measures. Therefore the the regularity follows from Lemma~\ref{lem:metricrho}. 
\end{rmk}
}

We now use the version of Ishii's lemma in Theorem \ref{ishii} to prove the comparison of viscosity solutions, which is the main result of the paper. 
The proof of this theorem is provided in Section \ref{s.proofcomp}.

\begin{thm}\label{thm:comp}
    Assume that $u$, $-v:\Theta\mapsto \R$ are bounded from above, uniformly continuous in $t$, either $u$ or $v$ is $d_F$-Lipschitz continuous functions in $(y,\mu)$, and $u$ (resp. $v$) is a viscosity subsolution (resp. supersolution) of the equation 
    \begin{align*}
        -\pa_t U(\theta)-h(\theta,U(\theta),D_y U(\theta),D_{yy}^2 U(\theta),D_\mu U(\theta,\cdot),D_{x\mu}U(\theta,\cdot),\Dc^2_{y\mu}U(\theta) ,\Hc U(\theta))=0
    \end{align*}
    for $\theta=(t,y,\mu)\in\mathring{\Theta}$. 
    
    Then, under Assumption~\ref{assumeh} $u(T,\cdot)\leq v(T,\cdot)$ implies that $u(t,\cdot)\leq v(t,\cdot)$ for all $t\in [0,T]$.
\end{thm}

\begin{rmk}
\begin{itemize}
    \item[(i)] In order to avoid technicalities and concentrate solely on the issues associated with the absence of a Hilbert space structure in the Wasserstein space, we demonstrate our comparison result under the assumption of a Lipschitz continuous generator $h$. We hypothesize that by incorporating the pertinent penalizations it is possible to relax this uniform continuity assumption.

    \item[(ii)] Assumption \eqref{eq:contpsi} is the counterpart of the assumption in \cite[Inequality (3.14)]{usersguide}. The derivatives $D_{\mu}\Psi(\mu^*)$ and $D^2_{y\mu}\Psi(\mu^*)$ admit explicit Fourier decomposition in Lemma \ref{lem:propsliced}. For the Hamilton-Jacobi-Bellman equation we study below, this assumption requires integrability and regularity of the Fourier transform of the data of the problem.
    \item[(iii)] { It is possible to show a comparison principle for semi-continuous solutions. However, without the Lipschitz assumption on solutions, one could only get $\frac{\rho_F^2(\mu^*,\tilde \mu^*)}{\epsilon} \to 0 $ in \eqref{eq:bound:thetas}. Then a stronger assumption on the Hamiltonian would be required to show a comparison principle. }
    
    \item[(iv)] The Lipschitz property of value functions w.r.t. the Fourier Wasserstein distance can also be obtained from the non-degeneracy of idiosyncratic noise; see \cite{2023arXiv231202324D,2023arXiv230804097M,2023arXiv230508423D}
    \end{itemize}
\end{rmk}

\section{Application to stochastic control with partial observation}\label{ss.controlproblem}
As an application, we show that the value function of the stochastic control problem with partial observation is the unique viscosity solution of a parabolic equation. Proofs are postponed to Section~\ref{s.proofe}.

Suppose that we have two independent Brownian motions $V,W$ of dimensions $d_1$ and $d_2$, and a compact closed control set $A \subset \R^d$. Take the coefficients $b:\R^d \times A \to \R^d$, $\sigma:\R^d \times A \to \R^{d\times d_1}$, and $\tilde\sigma:A \to \R^{d \times d_2}$. We consider the following stochastic differential equations

\begin{align*}
dX^{t,\mu,\alpha}_s&=b(X^{t,\mu,\alpha}_s,\alpha_s) \, ds + \sigma(X^{t,\mu,\alpha}_s,\alpha_s) \, dV_s + \tilde \sigma(\alpha_s) \, dW_s, \quad \text{$t \leq s \leq T$}, \\
X_t^{t,\mu,\alpha}&=\xi, 
\end{align*}
where $\xi$ is independent of $V,W$ with distribution $\mu$  and $\alpha_t \in A$ is an admissible control adapted to only $W$. Since $\xi$ is independent of $V,W$, it can be easily checked that the distribution of $(X_s^{t,\mu,\alpha},\alpha_s)$ is independent of the choice of $\xi$.

Let $m_t$ denote the conditional law  $\mathcal{L}(X_t \,|\, \mathcal{F}^{W}_t)$. Then $m_t$ satisfies the equation
\begin{align}\label{eq:conditionlaw}
    d m^{t,\mu,\alpha}_s(l) &=  m^{t,\mu,\a}_s (L^{\alpha_s} l) \,ds + m^{t,\mu,\a}_s(M^{\a_s} l) \, dW_s, \quad \text{ $t \leq s \leq T$}  \\
    m^{t,\mu,\alpha}_t &= \mu, \notag
\end{align}
where $l:\R^d \to \R$ is any $C^2$ test function and 
\begin{align*}
    L^a l(\cdot):=&  b(\cdot,a)^\top Dl(\cdot) +\frac{1}{2} Tr \left(  (\sigma\sigma^{\top}(\cdot,a)  +\tilde\sigma \tilde \sigma^{\top}(a)) D^2  l(\cdot)  \right),\\
    M^a l(\cdot):=&   \tilde\sigma(a)^\top Dl(\cdot). 
\end{align*}

Take $\mathcal{A}:=\{\alpha=(\alpha_s)_{0\leq s \leq T}: \alpha_s \in A \text{ is $\mathcal{F}_s^W$ measurable for all $s \in [0,T]$} \}$ to be the set of all admissible controls. Given a running cost $f:\R^d \times A \to \R$, and a terminal cost $g: \R^d \to \R$, we define the cost of control $\alpha \in \mathcal{A}$,
\begin{align*}
    J(t,\mu, \alpha):=\E \left[\int_t^T f(X^{t,\mu,\alpha}_s, \alpha_s) \, ds + g(X^{t,\mu,\alpha}_T)  \right].
\end{align*}
We aim at solving the following optimization problem 
\begin{align*}
    v(t,\mu)=\inf_{\alpha \in \mathcal{A}} J(t,\mu,\alpha). 
\end{align*}

\begin{rmk}
    (i) This is a simple version of the filtering problem, where we assume the controller observes the process $(Y_s)=(W_s)_{0 \leq s \leq T}$. If $f$ and $g$ are also functions of the observation process $(Y_s)$, one may still expect that the value function $v$ is the unique viscosity solution of some second-order PDE. Let us take the value function 
    $$
    v(t,y,\mu):= \inf_{\alpha \in \mathcal{A}} \E\left[\int_t^T f(X^{t,\mu,\alpha}_s, Y_s,\alpha_s) \, ds + g(X_T^{t,\mu,\alpha},Y_T) \right],$$ 
    and the generator
    \begin{align*}
K(a,y,\mu,X^{11},p,q,X^{12},X^{22}):=&\int  f(x,y,a)+ b(x,a)^\top p(x) +\frac{1}{2} Tr\left(q(x) \sigma \sigma^\top (x,a) \right)\, \mu(dx) \\
    &+\frac{1}{2} Tr(X^{11})+Tr(\tilde{\sigma}(a)\mathbf{1}^\top X^{12})+\frac{1}{2} Tr( \tilde \sigma \tilde \sigma^\top(a) X^{22}).
    \end{align*}
    With the notation $\theta=(t,y,\mu)$, it can be checked that the HJB equation of value function 
    \begin{align*}
        -\pa_t v(\theta)&= \inf_{a \in A} K\left(a,y,\mu, D_{yy}^2 v(\theta), D_{\mu} v(\theta),D_{x \mu} v(\theta), \mathcal{D}^2_{y \mu}(\theta), \mathcal{H}v(\theta)\right), \\
        v(T,y,\mu)&= \int g(x,y) \, \mu(dx). 
    \end{align*}
fits in the form of Theorem~\ref{thm:comp}.  Here for simplicity of notation, we prove the case where the coefficients are not functions of $y$, and leave the general case to interested readers. 

(ii) In \cite{JMLR:v24:22-1001}, the authors characterize the asymptotic behavior of a learning problem by a second-order PDE on Wasserstein space, which fits in the form of \eqref{eq:introparabolic} and satisfies Assumption~\ref{assumeh}.

{(iii) We can also allow $b,f,\sigma,\tilde \sigma$ to be $\mu$ dependent with additional assumption on its dependence in $\mu$.  

(vi) We conjecture that our uniqueness result can be applied to first-order equations of the McKean-Vlasov control studied in \cite{soner2022viscosity} where the control $\a$ is allowed to be a feedback control dependent on $X$. In this case, similar to \cite{soner2022viscosity}, the admissibility condition of controls must require the regularity $x\mapsto b(x,\a(x)).$
}

\end{rmk}

Recall that $\lambda=d+7$ throughout the paper. 

\begin{assumption}\label{assumption1}
\begin{itemize}
    \item[(i)] The functions $b,\sigma,\tilde\sigma, f,g$ are bounded and continuous over their domains. 
    \item[(ii)] $b(\cdot,a),\sigma(\cdot,a)$ are $\lambda$-times differentiable with bounded derivatives uniform in $a$, i.e., 
    $$\sup_{a \in A,k=0,\dotso, \lambda} \lVert b^{(k)}(\cdot,a)\rVert_{\infty}+\sup_{a \in A}\lVert \sigma^{(k)}(\cdot,a) \rVert_{\infty}< \infty,$$ 
    and also 
    $$\sup_{a \in A,i,j=0,\dotso, d}||x^i b^j(x,a)||_\lambda+||b^j(x,a)||_\lambda+||f(\cdot,a)||_{\lambda}<\infty.$$
    \item[(iii)] Derivatives of cost functions $f,g$ are of exponential decay, i.e., there exist some positive constants $K,c$ such that $\sup_{a \in A}|f^{(k)}(x,a)|+|g^{(k)}(x)| \leq Ke^{-c|x|}$ for $k=0,\dotso, \lambda$.
\end{itemize}
\end{assumption}

\begin{lemma}\label{lem:measureflow_estimate}
    Under Assumption~\ref{assumption1}, the solution to \eqref{eq:conditionlaw} is pathwise continuous and it satisfies 
    \begin{align*}
        \E[W_2(m^{t,\mu,\alpha}_s, \mu)^2]+\E \left[ \lVert m_s^{t,\mu,\alpha}-\mu \rVert^2_{\lambda} \right] \leq L|s-t|,
    \end{align*}
    where $L$ is some positive constant depending only on coefficients $b$, $\sigma$, and $\tilde\sigma$. 
\end{lemma}

\begin{proposition}\label{prop:Lipschizvalue}
    Under Assumption~\ref{assumption1}, the value function $v: [0,T] \times \Pc_2(\R^d) \to \R$ is bounded and Lipschitz continuous with respect to $\rho_F$, uniformly continuous in time and satisfies the dynamic programming principle 
\begin{align*}
    v(t,\mu) &= \inf_{\alpha \in \mathcal{A}} \inf_{\tau \in \mathcal{T}_t^T} \E \left[\int_t^\tau f(X^{t,\mu,\alpha}_s,\alpha_s) \,ds + v(\tau, m^{t,\mu,\alpha}_{\tau})  \right] \\
    &=\inf_{\alpha \in \mathcal{A}} \sup_{\tau \in \mathcal{T}_t^T} \E \left[\int_t^\tau f(X^{t,\mu,\alpha}_s,\alpha_s) \,ds + v(\tau, m^{t,\mu,\alpha}_{\tau})  \right],
\end{align*}
where $\mathcal{T}_t^T$ is the set of $[t,T]$-valued stopping times $\tau$ with respect to the filtration $(\mathcal{F}^W_s)_{s \in [t,T]}$
\end{proposition}

Another crucial ingredient to show the viscosity property of value functions is It\^{o}'s formula for measure-valued process. Our argument relies on the It\^{o}-Wentzell formula, see e.g.  \cite{MR0995291}, and the fact that the second-order generator of $m_t$ is given by the Wasserstein Hessian $\mathcal{H}$, and hence we require weaker regularity.

\begin{proposition}[It\^{o}'s formula for measure valued SDE]\label{lem:Ito}
Suppose $\psi: [0,T] \times \Pc_2(\R^d) \to \R$ is partially $C^2$-regular. Then we have that 
\begin{align*}
    d\psi(s, m_{s})=& \pa_t \psi(s,m_{s} ) \, ds + \int b(x, \alpha_s)^\top D_{\mu} \psi(s,m_s)(x)\, m_s(dx) \, ds  \\
    &+ \frac{1}{2} \int  Tr \left(D_{x \mu} \psi(s,m_s)(x) \sigma\sigma^\top(x,\alpha_s)\right) m_s(dx) \, ds \\
    &+ \frac{1}{2} Tr (\mathcal{H} \psi(s,m_s) \tilde\sigma \tilde\sigma^\top(\alpha_s)) \, ds+ D_{\mu} \psi(s,m_s)([m_s])\tilde \sigma(\alpha_s) \, dW_s.
\end{align*}
\end{proposition}

Let us define for $(a,\mu,p,q,M) \in A  \times \Pc_2(\R^d) \times B^d_{\mathfrak{q}} \times B^{d \times d}_{\mathfrak{q}} \times \mathcal{S}^d $
\begin{align*}
    K(a,\mu,p,q,M):=& \int  f(x,a)+ b(x,a)^\top p(x) +\frac{1}{2} Tr\left(q(x) \sigma \sigma^\top (x,a) \right)\, \mu(dx) \\
    &+\frac{1}{2} Tr( \tilde \sigma \tilde \sigma^\top(a) M).
\end{align*}

\begin{thm}[Viscosity solution]\label{thm:viscosity_property}
    Under Assumption~\ref{assumption1}, the value function is the unique viscosity solution of the equation 
    \begin{align*}
        -\pa_t v(t,\mu)=& \inf_{a \in A} K(a, \mu, D_{\mu} v(t,\mu), D_{x \mu} v(t,\mu), \mathcal{H}v(t,\mu)) \\
        v(T,\mu)=&\mu(g). 
    \end{align*}
\end{thm}

\section{Auxiliary Results}

\subsection{Differentiability of the metrics}

We will use the following lemma to compute derivatives in the Wasserstein space. 
\begin{lemma}\label{lem:u0}
    Let $u:\Pc_2(\R^d)\mapsto \R$ be first-order differentiable, i.e., $u \in C^1$. Define $u_0:\Pc_2(\R^d)\mapsto \R$ by the equation 
    $u_0(\mu):=u(\Sc_0(\mu))=u((I_d-m(\mu))\sharp \mu)$. 
    Then, $u_0$ is also differentiable, and for all $x \in \R^d$ we have
    $$D_\mu u_0(\mu,x)=D_\mu u(\Sc_0(\mu),x-m( \mu))-D_\mu u(\Sc_0(\mu),[\Sc_0(\mu_0)]).$$
    Therefore, the derivative of $\mu\mapsto F_k(\Sc_0(\mu))$ is
    \begin{align*}
        D_\mu (F_k(\Sc_0(\mu)))(x)&=i ke^{ik\cdot m(\mu)}\left(  \int f_k(z)\mu(dz)-f_k(x)\right)= i k (F_k(\Sc_0(\mu))-f_k(x-m(\mu))).
    \end{align*}
\end{lemma}
\begin{proof}
    We fix $\mu\in \Pc_2(\R^d)$, $h$ a continuous and bounded function and define $\mu^\e:=(I_d+\e h)_\sharp \mu$. Then, we have the equality of laws 
    \begin{align*}
    \Sc_0(\mu^\e)&=(I_d- m(\mu)- \e h([\mu]))_\sharp \mu^\e \\
    &=(I_d+\e h(\cdot)- m(\mu)-\e h([\mu]))_{\sharp} \mu \\
    &=(I_d+\e(h(\cdot+m(\mu))-h([ \mu])))_\sharp \Sc_0(\mu).
    \end{align*}
    By direct computation we have the following limit
    \begin{align*}
       \lim\limits_{\e \to 0} \frac{u_0(\mu^\e)-u_0(\mu)}{\e}&=\lim\limits_{\e \to 0}\frac{u\left((I_d+\e(h(\cdot+ m( \mu))-h([ \mu])))_\sharp \Sc_0(\mu)\right)-u(\Sc_0(\mu))}{\e}\\
        &=\int D^\top_\mu u(\Sc_0(\mu),x)(h(x+m( \mu))-h([ \mu])) \, \Sc_0(\mu)(dx)\\
        &=\int D^\top_\mu u(\Sc_0(\mu),x-m( \mu))h(x) \, \mu(dx)- D^\top_\mu u(\Sc_0(\mu),[\Sc_0(\mu)]])h([ \mu])\\
        &=\int \big(D_\mu u(\Sc_0(\mu),x-m(\mu))- D_\mu u(\Sc_0(\mu),[\Sc_0(\mu)])\big)^\top h(x) \, \mu(dx)
    \end{align*}
    which concludes the proof. 
\end{proof}

Denote the canonical basis of $\R^d$ by $( e_1,\ldots, e_d)$. Then one can easily show that $\mu \mapsto V(\mu)$ is $L$-differentiable and for all $i,j$
$$D_\mu V^{i,j}(\mu,x)=(x^j-m^j(\mu)) e_i+(x^i-m^i( \mu)) e_j,$$
where $V^{i,j}$ stands for the $(i,j)$-entry of matrix $V$. With these preparation, we compute the derivatives of $\rho_F^2(\mu,\nu)$. For a function $f:\Pc_2(\R^d) \times \Pc_2(\R^d) \to \R$, we define its partial Hessian w.r.t. $(\mu,\nu)$ via 
\begin{align*}
    \mathcal{H}_{\mu,\nu} f(\mu,\nu)= \left(\frac{d^2}{dx_idx_j}\Big|_{x_1=x_2=0} f((I_d+x_1)_{\sharp}\mu,(I_d+x_2)_{\sharp}\nu) \right)_{1\leq i,j \leq 2} \in \R^{2d \times 2d}. 
\end{align*}

\begin{lemma}\label{lem:propsliced}
Let $\mu,\nu\in \Pc_2(\R^k)$. 
Then $(\mu,\nu) \mapsto \rho^2_F(\mu,\nu)$ is twice-differentiable, and 
\begin{align*}
    D_{\mu}\rho^2_F(\mu,\nu,x)=&2 (m( \mu)-m (\nu))+{4(V(\mu)-V(\nu))(x-m(\mu))}\\
    &+2\int\frac{Re\left(i k\left(  F_k(\Sc_0(\mu))-f_k(x-m(\mu))\right)(F_k(\Sc_0(\mu))-F_k(\Sc_0(\nu)))^*\right)}{(1+|k|^2)^\lambda}dk\\
    D_{\mu x} \rho^2_F(\mu,\nu,x)=& {4 (V(\mu)-V(\nu))}\\
    &-2\int\frac{kk^\top Re\left(f_k(x-m(\mu))(F_k(\Sc_0(\mu))-F_k(\Sc_0(\nu)))^*\right)}{(1+|k|^2)^\lambda}dk\\   
\Hc_{\mu,\nu}\rho^2_F(\mu,\nu)=&2\begin{pmatrix}I_d&-I_d\\-I_d&I_d\end{pmatrix}.
\end{align*}
Furthermore, $D_\mu \rho_F^2,D_{x \mu }\rho_F^2$ are continuous in $\mu$ in any of the equivalent topologies considered.
\end{lemma}
\begin{proof}
The first two equations directly follow from Lemma~\ref{lem:u0}. The last is due to the fact that for $c \in \R$
\begin{align*}
      \rho_F^2((I_d+c)_{\sharp}\mu ,\nu)-\rho_F^2(\mu ,\nu)=\rho_F^2(\mu, (I_d-c)_{\sharp}\nu)-\rho_F^2(\mu ,\nu)= |\bar \mu -\bar \nu +c|^2-|\bar \mu -\bar \nu |^2,
    \end{align*}
    and 
    \begin{align*}
    D_{\mu} \rho_F^2(\mu,\nu, [\mu])=\frac{d}{dc} \rho_F^2((I_d+c)_{\sharp}\mu,\nu) |_{c=0}, \\
    \Hc_{\mu} \rho_F^2(\mu,\nu)=\frac{d^2}{dc^2} \rho_F^2((I_d+c)_{\sharp}\mu,\nu) |_{c=0}. 
    \end{align*}

    Since $ \rho_F, W_2$ induce the same topology on $\Pc_2(\R^d)$, we prove the continuity of $\mu \mapsto (D_{\mu} \rho_F^2(\mu,\nu,x), D_{x \mu } \rho_F^2(\mu,\nu,x))$ with respect to the $W_2$ metric. Note that $\mu \mapsto (m(\mu), V(\mu), \Sc_0(\mu))$ is clearly continuous and the Fourier transforms and basis $F_k(\nu), f_k(x)$ are uniformly bounded for all $\nu \in \Pc_2(\R^d), k,x \in \R^d$. Our claim simply follows from the continuity of $\mu \mapsto F_k(\mu)$ and the dominated convergence theorem. 
\end{proof}

In the proof of the main theorems, we will use the following auxiliary function to reduce functions on Wasserstein space to the ones on $\R^d$. For $\mu,\eta,\nu \in \Pc_2(\R^d)$, denote
\begin{align*}
\Lc(\mu,\eta,\nu):=&{2}Tr(V^\top(\Sc_0(\mu))(V(\Sc_0(\eta))-V(\Sc_0(\nu)))) \\
&+{2}\int \frac{Re(F_k(\Sc_0(\mu))(F_k(\Sc_0(\eta))-F_k(\Sc_0(\nu)))^*)}{(1+|k|^2)^\lambda}dk. 
\end{align*}
Here are some properties of $\Lc$ whose proof is almost the same as in Lemma~\ref{lem:propsliced} and thus omitted. 

\begin{lemma}\label{lem:L}
    Let $\mu,\eta,\nu \in \Pc_2(\R^d)$. Then $\mu \mapsto \Lc(\mu,\eta,\nu)$ is twice-differentiable, and 
    \begin{align*}
        D_{\mu}\Lc(\mu,\eta,\nu,x)=&{2}\sum_{i,j}(V^{i,j}( \Sc_0(\eta))-V^{i,j}(\Sc_0(\nu)))\big((x^j-m^j( \mu)) e_i+(x^i-m^i( \mu)) e_j\big) \\
        &+{2}\int \frac{Re(i k (F_k(\Sc_0(\mu))-f_k(x-m( \mu)))(F_k( \Sc_0(\eta))-F_k(\Sc_0(\nu)))^*)}{(1+|k|^2)^\lambda}dk,\\
  D_{x \mu }\Lc(\mu,\eta, \nu,x) =&{2}\sum_{i,j} (V^{i,j}(\Sc_0(\eta))-V^{i,j}(\Sc_0(\nu)))\big( e_ie_j^\top+ e_je_i^\top\big) \\
    &-{2}\int\frac{kk^\top Re\left(f_k(x-m(\mu))(F_k(\Sc_0(\eta))-F_k(\Sc_0(\nu)))^*\right)}{(1+|k|^2)^\lambda}dk.
    \end{align*}
Furthermore, $\mu \mapsto (D_{\mu} \Lc(\mu,\eta\,\nu,x), D_{x \mu }\Lc(\mu,\eta,\nu,x))$ is continuous w.r.t. any of the equivalent topology considered. 
\end{lemma}

\subsection{Variational principle}
Thanks to \cite[Lemma 2]{auricchio2020equivalence} and \cite[Lemma 4.4]{cosso2021master}, 
$\rho_\sigma$ is a gauge-type function on the complete metric space $(\Pc_2(\R^d),W^{(\sigma)}_2)$. 
Recall the definition of $d_\sigma$ in \eqref{eq:defdsigma}.
Thanks to \cite[Lemma 4.3 and Lemma 4.4]{cosso2021master}, 
the mapping 
$$((\theta_1,\tilde \theta_1),(\theta_2,\tilde \theta_2))\in (\Theta^2)^2\mapsto d_\sigma^2(\theta_1,\theta_2)+d_\sigma^2(\tilde \theta_1,\tilde \theta_2)$$
is a gauge-type function on the product space $\Theta^2$ where $\Pc_2(\R^d)$ is endowed with the metric $W^{(\sigma)}_2$. Being topologically equivalent to $W_2$ and $\rho_F$, these metric admit the same set of semi-continuous functions. Thus, we have the following version of the Borwein-Preiss variational principle whose proof is the same as in \cite[Theorem 4.5]{cosso2021master} up to the addition of the $y$ variable and the doubling of all variables. The upper bound of the derivatives can be easily derived from \cite[(4.6),(4.7)]{cosso2021master} and Lemma~\ref{lem:u0}.
\begin{lemma}\label{lem:var}
Fix $\kappa>0$ and let $G:\Theta^2\mapsto \R$ be an upper semicontinuous function. Assume that  
$$\sup_{(\theta, \tilde \theta)\in \Theta^2}G(\theta, \tilde \theta)<\infty.$$
and let $(\theta^0, \tilde \theta^0)\in \Theta^2$ so that 
$$\sup_{(\theta, \tilde \theta)\in \Theta^2}G(\theta, \tilde \theta)-\kappa<G(\theta^0, \tilde \theta^0).$$
Then, for every $\delta>0$, there exists $(\theta^*, \tilde \theta^*)$ and 
$\{(\theta^j, \tilde \theta^j):j\geq 1\}$ such that denoting 
$$\phi_\delta(\theta):=\sum_{j=0}^\infty\frac{d_\sigma^2(\theta,\theta^j)}{2^j}\mbox{ and }\tilde \phi_\delta(\tilde \theta):=\sum_{j=0}^\infty\frac{d_\sigma^2(\tilde \theta,\tilde \theta^j)}{2^j}$$
we have
\begin{itemize}
    \item $d_\sigma^2(\theta^*,\theta^j)+d_\sigma^2(\tilde \theta^*,\tilde \theta^j)\leq \frac{\kappa}{\delta 2^j}$ for $j\geq 0$;
    \item$
        G(\theta^0, \tilde \theta^0)\leq G(\theta^*, \tilde \theta^*)- \delta(\phi_\delta(\theta^*)+\tilde \phi_\delta(\tilde \theta^*))$;
    \item $(\theta, \tilde \theta)\in \Theta^2\mapsto G(\theta, \tilde \theta)- \delta  (\phi_\delta(\theta)+\tilde \phi_\delta(\tilde \theta))$ admits a strict global maximum at $(\theta^*, \tilde \theta^*).$
\end{itemize}
Additionally $d_\sigma$ satisfies the following derivative bounds 
\begin{align*}
|D_\mu d_\sigma^2(\theta,\tilde \theta,x)|&\leq C_d\left(\sigma +\sigma^{-1}\left(|x|^2+|m(\mu)|^2+\int |x|^2 \, \Sc_0(\mu)(dx)\right)\right)+2|m(\mu)-m(\tilde \mu)|,\\
    |D_{x\mu} d_\sigma^2(\theta,\tilde \theta,x)|&\leq C_d(1 +\sigma^{-2}(|x|^2+|m(\mu)|^2)),
\end{align*}
for a constant $C_d$ that only depends on the dimension $d$. 

\end{lemma}

\section{Proof of Ishii's Lemma}\label{s.proofishii}

The proof is divided into several steps. Let us briefly sketch the idea. In Step 1, making use of the product space structure $\mathcal{P}_2(\R^d)=\R^d \times \mathcal{P}_{2,0}(\R^d)$, we reduce sub/super solutions $u$ and $v$ to finite dimensional auxiliary functions. In Step 2, we invoke the finite dimensional Ishii's Lemma to get second-order finite dimensional jets of auxiliary functions. Then we can go all the way to obtain all the jets of $u$ and $v$. Here a tricky point is that when defining auxiliary functions, we take supreme/infimum among all probability measures with fixed barycenters. However, due to the non-compactness of Wasserstein space, the supreme and infimum cannot be obtained, and thus we invoke the variational principle in Step 3, and derive a limit in Step 4. 

\begin{proof}[Proof of Theorem~\ref{ishii}]
 {\it Step 1: Reducing to finite dimensional Ishii's lemma.}
Define the functions
        \begin{align*}
         u_1(t,y,\mu)&:= u(t,y,\mu) -\frac{ \a}{2}\Psi(\mu), \\
         v_1(t, y,\mu)&:=v(t, y,\mu)+\frac{\a}{2}\tilde \Psi(\mu).
    \end{align*}
    Given any Hilbert space $(L,\langle \cdot \rangle )$ and $a,b,c,d \in L$, it is straightforward that 
    \begin{align*}
    |a-b|^2+|c-d|^2-2\langle a-b,c-d\rangle = |a-b-(c-d)|^2 \leq 2 (|a-c|^2+|b-d|^2).
    \end{align*}
    Therefore we have 
    \begin{align}
       &|F_k(\Sc_0(\mu))-F_k(\Sc_0(\tilde \mu))|^2+|F_k(\Sc_0(\mu^*))-F_k(\Sc_0(\tilde \mu^*))|^2\notag\\
       &\quad\quad\leq2(|F_k(\Sc_0(\mu))-F_k(\Sc_0(\mu^*))|^2+|F_k(\Sc_0(\tilde \mu))-F_k(\Sc_0(\tilde \mu^*))|^2)\notag\\
       &\quad\quad+ 2Re((F_k(\Sc_0(\mu))-F_k(\Sc_0(\tilde \mu)))(F_k(\Sc_0(\mu^*))-F_k(\Sc_0(\tilde \mu^*))^*))\label{ineq:1}
       \end{align}
       and 
       \begin{align}
       &Tr((M-N)^\top(M-N))+ Tr((M^*-N^*)^\top(M^*-N^*))\notag\\
       &\quad\quad\leq 2( Tr((M-M^*)^\top(M-M^*))+Tr((N-N^*)^\top(N-N^*)))\notag\\
        &\quad\quad+ 2Tr((M-N)^\top(M^*-N^*)))\label{ineq:2}
    \end{align}
    valid for any symmetric matrices $M=V(\mu),N=V(\tilde \mu),M^*=V(\mu^*),N^*=V(\tilde \mu^*)$. Hence it holds that 
\begin{align*}
    &2\rho^2_F(\Sc_0(\mu),\Sc_0(\mu^*))+2\rho^2_F(\Sc_0(\tilde \mu),\Sc_0(\tilde \mu^*))+\Lc(\Sc_0(\mu),\Sc_0(\mu^*),\Sc_0(\tilde \mu^*))-\Lc(\Sc_0(\tilde \mu),\Sc_0(\mu^*),\Sc_0(\tilde \mu^*))\\
    &\geq \rho^2_F(\Sc_0(\mu),\Sc_0(\tilde \mu))+\rho^2_F(\Sc_0(\mu^*),\Sc_0(\tilde \mu^*)).
\end{align*}
Thus, we have the inequality, 
\begin{align*}
     & u_1(t,y,\mu)-v_1(\tilde t,\tilde y,\tilde \mu)-\frac{\a}{2}\left(|y-\tilde y|^2+|m(\mu)- m(\tilde\mu)|^2+|t-\tilde t|^2\right)\\
     &= u(t,y,\mu)-v(\tilde t,\tilde y,\tilde \mu)-\frac{\a}{2}\left(|y-\tilde y|^2+|m(\mu)- m(\tilde\mu)|^2+|t-\tilde t|^2\right)\\
     &-\frac{ \a}{2}\left(2\rho_F^2(\Sc_0(\mu),\Sc_0(\mu^*))+2\rho^2_F(\Sc_0(\tilde \mu),\Sc_0(\tilde \mu^*))\right)\\
     &-\frac{ \a}{2}\left(\Lc(\Sc_0(\mu),\Sc_0(\mu^*),\Sc_0(\tilde \mu^*))-\Lc(\Sc_0(\tilde \mu),\Sc_0(\mu^*),\Sc_0(\tilde \mu^*))\right)\\
     &\leq u(t,y,\mu)-v(\tilde t,\tilde y,\tilde \mu)-\frac{\a}{2}\left(|y-\tilde y|^2+|m(\mu)- m(\tilde\mu)|^2+|t-\tilde t|^2\right)\\
     &-\frac{\a}{2}(\rho^2_F(\Sc_0(\mu),\Sc_0(\tilde \mu))+\rho^2_F(\Sc_0(\mu^*),\Sc_0(\tilde \mu^*)))\\
     &=u(\theta)-v(\tilde \theta)-\frac{\a}{2} d_F^2(\theta,\tilde \theta)-\frac{\a}{2}\rho^2_F(\Sc_0(\mu^*),\Sc_0(\tilde \mu^*)).
\end{align*}
Additionally, for $(\mu,\tilde \mu,M,N)=(\mu^*,\tilde \mu^*,M^*,N^*)$ the inequalities \eqref{ineq:1}-\eqref{ineq:2} are equalities. 
Thus, the function
\begin{align}\label{eq:optu1v1}
u_1(t,y,\mu)-v_1(\tilde t,\tilde y,\tilde \mu)-\frac{\a}{2}\left(|y-\tilde y|^2+|m(\mu)- m(\tilde\mu)|^2+|t-\tilde t|^2\right)
\end{align}
admits the same strict global maximum at $((t^*,y^*,\mu^*),(\tilde t^*,\tilde y^*,\tilde \mu^*))$ as $$u(\theta)-v(\tilde \theta)-\frac{\a}{2} d_F^2(\theta,\tilde \theta) .$$
Define the mappings 
\begin{align}
    (t,y,z)\in [0,T]\times (\R^d)^2\mapsto  u_2(t,y,z)=\sup_{\mu\in \Pc_{2,0}(\R^d)} u_1(t,y,(I_d+z)\sharp\mu)\label{eq:defu2}\\
    (t,y,z)\in [0,T]\times (\R^d)^2\mapsto v_2(t,y,z)=\inf_{\mu\in \Pc_{2,0}(\R^d)} v_1(t,y,(I_d+z)\sharp\mu)\label{eq:defv2}
\end{align}
where we recall that $\Pc_{2,0}(\R^d)\subset \Pc_{2}(\R^d)$ is the set of measures with $0$ mean and finite variance. By the maximality of $ u_1(t,y,\mu)-v_1(\tilde t,\tilde y,\tilde \mu)-\frac{\a}{2}\left(|y-\tilde y|^2+| m(\mu)- m(\tilde\mu)|^2+|t-\tilde t|^2\right)$ at $((t^*,y^*,\mu^*),(\tilde t^*,\tilde y^*,\tilde \mu^*))$, the function 
\begin{align}\label{eq:optfinite}
      u_2(t,y,z)-v_2(\tilde t,\tilde y,\tilde z)-\frac{\a}{2}\left(|y-\tilde y|^2+|z-\tilde z|^2+|t-\tilde t|^2\right)
\end{align}
 admits a global maximum at $((t^*,y^*,m(\mu^*)),(\tilde t^*,\tilde y^*,m( {\tilde {\mu}}^*)))$. 
Denoting $u_2^*$ the upper semicontinuous envelope of $u_2$ and $v_{2,*}$ the lower semicontinuous envelope of $v_2$, this maximality implies that 
\begin{align}\label{eq:optfinitestar}
      u_2^*(t,y,z)-v_{2,*}(\tilde t,\tilde y,\tilde z)-\frac{\a}{2}\left(|y-\tilde y|^2+|z-\tilde z|^2+|t-\tilde t|^2\right)
\end{align}
 admits a global maximum at $(((t^*,y^*,m(\mu^*)),(\tilde t^*,\tilde y^*,m( {\tilde {\mu}}^*)))$ and 
 $$u_2^*(t^*,y^*,m(\mu^*))-v_{2,*}(\tilde t^*,\tilde y^*,m( {\tilde {\mu}}^*))= u_2(t^*,y^*,m(\mu^*))-v_2(\tilde t^*,\tilde y^*,m( {\tilde {\mu}}^*)).$$


 \vspace{4pt}
 
 {\it Step 2: Applying finite dimensional Ishii's lemma.}
 By Ishii's Lemma in \cite[Theorem 3.2]{usersguide} applied to \eqref{eq:optfinitestar}, for $\e>0$,
there exists a sequence indexed by $n$
$$(t^*_n,\tilde t^*_n,y^*_n,\tilde y^*_n,z^*_n,\tilde z^*_n,q_n^*,\tilde q_n^*,p_n^*,\tilde p_n^*,X_n^*,\tilde X_n^*)\in (\R)^2\times (\R^d)^4\times \R^2\times (\R^{2d})^2\times (\Sc_{2d})^2$$ 
so that 
\begin{align}\label{maxu2v2}
    &u_2^*(t,y,z)-v_{2,*}(\tilde t,\tilde y,\tilde z)-q_n^*(t-t^*_n)-\tilde q_n^*(\tilde t-\tilde t^*_n)-(p^*_n)^\top \begin{pmatrix}y-y^*_n\\z-z^*_n\end{pmatrix}
   \\
    & -(\tilde p^*_n)^\top\begin{pmatrix}\tilde y-\tilde y^*_n\\\tilde z-\tilde z^*_n\end{pmatrix}-\frac{1}{2}\begin{pmatrix}
        y-y^*_n\\z-z^*_n
    \end{pmatrix}^\top X^*_n\begin{pmatrix}
        y-y^*_n\\z-z^*_n
    \end{pmatrix}
        -\frac{1}{2}\begin{pmatrix}
         \tilde y-\tilde y^*_n\\\tilde z-\tilde z^*_n
    \end{pmatrix}^\top \tilde X^*_n\begin{pmatrix}
    \tilde y-\tilde y^*_n\\\tilde z-\tilde z^*_n
    \end{pmatrix}\notag
\end{align}
admits a local maximum at $((t^*_n,y^*_n,z^*_n),(\tilde t^*_n,\tilde y^*_n,\tilde z^*_n))$, and we have the following bounds and limits as $n\to \infty$,
\begin{align}\label{eq:tstarn}
    ((t^*_n,y^*_n,z^*_n),(\tilde t^*_n,\tilde y^*_n,\tilde z^*_n)) &\to\left(((t^*,y^*,m(\mu^*)),(\tilde t^*,\tilde y^*,m(\tilde \mu^*))\right),\\
    (u_2^*(t^*_n,y^*_n,z^*_n),v_{2,*}(\tilde t^*_n,\tilde y^*_n,\tilde z^*_n))&\to (u_2^*(t^*,y^*,m(\mu^*)),v_{2,*}(\tilde t^*,\tilde y^*,m(\tilde \mu^*))),\notag\\
    &=(u_2(t^*,y^*,m(\mu^*)),v_2(\tilde t^*,\tilde y^*,m( {\tilde {\mu}}^*)))\notag\\
    (q_n^*,q_n^*+\tilde q_n^*,p_n^*,p_n^*+\tilde p_n^*,X^*_n,\tilde X^*_n) &\to\left(\a(t^*-\tilde t^*),0,\a\begin{pmatrix}
    y^*-\tilde y^*\\ m( \mu^*- {\tilde {\mu}}^*)
\end{pmatrix},0,X^*,\tilde X^*\right), \notag\\
\label{eq:orderXn}
    -\left(\frac{1}{\e}+2\a\right)I_{4d}\leq 
&\begin{pmatrix}
    X^*&0\\0&\tilde X^*
\end{pmatrix}
\leq 
(\a+2\e\a^2)
\begin{pmatrix}
    I_{2d}&-I_{2d}\\-I_{2d}&I_{2d}
\end{pmatrix}.
\end{align}

\vspace{4pt}

{\it Step 3: Application of the variational principle.}
 Note that the set $\{(t_n^*,y^*_n,z^*_n):n\in \N\}\cup \{(\tilde t_n^*,\tilde y^*_n,\tilde z^*_n):n\in \N\}$ is bounded. Thus, taking  $$R>\sup_n\sqrt{|t_n^*|^2+|y^*_n|^2+|z^*_n|}+\sqrt{|
 \tilde t_n^*|^2+|\tilde y^*_n|^2+|\tilde z^*_n|}$$ large enough, thanks to the continuity \eqref{cond:unifc0} which is a uniform continuity on bounded sets, and the optimality \eqref{eq:optu1v1}, there exists a continuous modulus of continuity $\omega_R$ so that for all $(t_0,y_0,z_0,\tilde t_0,\tilde y_0,\tilde z_0)$ with $$\sqrt{|t_0|^2+|y_0|^2+|z_0|}+\sqrt{|\tilde t_0|^2+|\tilde y_0|^2+|\tilde z_0|}\leq R$$ we have
\begin{align*}
      & u_1(t_0,y_0,(I_d+z_0)\sharp\Sc_0(\mu^*))-v_1(\tilde t_0,\tilde y_0,(I_d+\tilde z_0)\sharp\Sc_0(\tilde \mu^*))\\
      &-\frac{\a}{2}\left(|y_0-\tilde y_0|^2+|z_0- \tilde z_0|^2+|t_0-\tilde t_0|^2\right)\\
      & \geq u_1(t^*,y^*,(I_d+m(\mu^*))\sharp\Sc_0(\mu^*))-v_1(\tilde t^*,\tilde y^*,(I_d+m(\tilde z^*))\sharp\Sc_0(\tilde \mu^*))\\
      &-\frac{\a}{2}\left(|y^*-\tilde y^*|^2+|m(\mu^*)- m(\tilde \mu^*)|^2+|t^*-\tilde t^*|^2\right)\\
      &-\omega_R(| t^*- t_0|+|y^*-y_0|+|m(\mu^*)- z_0|+|\tilde t^*-\tilde t_0|+|\tilde y^*-\tilde y_0|+|m(\tilde \mu^*)- \tilde z_0|)\\     
      &= \sup_{\theta,\tilde \theta\in \Theta} u_1(t,y,\mu)-v_1(\tilde t,\tilde y,\tilde \mu)-\frac{\a}{2}\left(|y-\tilde y|^2+|m(\mu)- m(\tilde \mu)|^2+|t-\tilde t|^2\right)\\
      &-\omega_R(| t^*- t_0|+|y^*-y_0|+|m(\mu^*)- z_0|+|\tilde t^*-\tilde t_0|+|\tilde y^*-\tilde y_0|+|m(\tilde \mu^*)- \tilde z_0|)\\     
      &\geq \sup_{\mu,\tilde \mu\in \Pc_{2,0}(\R^d)} u_1(t_0,y_0,(I_d+z_0)\sharp\mu)-v_1(\tilde t_0,\tilde y_0,(I_d+\tilde z_0)\sharp\tilde \mu)\\
      &-\frac{\a}{2}\left(|y_0-\tilde y_0|^2+|z_0- \tilde z_0|^2+|t_0-\tilde t_0|^2\right)\\
      &-\omega_R(| t^*- t_0|+|y^*-y_0|+|m(\mu^*)- z_0|+|\tilde t^*-\tilde t_0|+|\tilde y^*-\tilde y_0|+|m(\tilde \mu^*)- \tilde z_0|)     
\end{align*}
which by continuity implies that for all $(t_0,y_0,z_0,\tilde t_0,\tilde y_0,\tilde z_0)$ with $$\sqrt{|t_0|^2+|y_0|^2+|z_0|}+\sqrt{|\tilde t_0|^2+|\tilde y_0|^2+|\tilde z_0|}\leq R,$$ we have
\begin{align*}
      & u_1(t_0,y_0,(I_d+z_0)\sharp\Sc_0(\mu^*))-v_1(\tilde t_0,\tilde y_0,(I_d+\tilde z_0)\sharp\Sc_0(\tilde \mu^*))\\
      &+\omega_R(|y^*-y_0|+|m(\mu^*)- z_0|+|\tilde t^*-\tilde t_0|+|\tilde y^*-\tilde y_0|+|m(\tilde \mu^*)- \tilde z_0|+|\tilde t^*-\tilde t_0|)\\
      &\geq  u_2^*(t_0,y_0,z_0)-v_{2,*}(\tilde t_0,\tilde y_0,\tilde z_0).
\end{align*}
Thus, denoting 
\begin{align}\notag
(\theta^{0,n},\tilde \theta^{0,n})&:=((t^{0,n},y^{0,n},\mu^{0,n}),(\tilde t^{0,n},\tilde y^{0,n},\tilde \mu^{0,n}))\\
&:=((t^*_n,y^*_n,(I_d+z^*_n)\sharp\Sc_0(\mu^{*})),(\tilde t^*_n,\tilde y^*_n,(I_d+\tilde z^*_n)\sharp\Sc_0(\tilde\mu^{*}) )  )\label{eq:0ndef}\\
C_n&:=\omega_R(|y^*-y^*_n|+|m(\mu^*)- z^*_n|+|\tilde t^*-\tilde t^*_n|+|\tilde y^*-\tilde y^*_n|+|m(\tilde \mu^*)- \tilde z^*_n|+|\tilde t^*-\tilde t^*_n|)\notag
\end{align}
which goes to $0$ thanks to \eqref{eq:tstarn}, 
the function 
\begin{align*} 
G_n(\theta,\tilde \theta)&:= u_1(t,y,\mu )-v_1(\tilde t,\tilde y, \tilde \mu)-q_n^*(t-t^*_n)-\tilde q_n^*(\tilde t-\tilde t^*_n)
    \\
    &-(p^*_n)^\top \begin{pmatrix}y-y^*_n\\m( \mu)-z^*_n\end{pmatrix}
    -(\tilde p^*_n)^\top\begin{pmatrix}\tilde y-\tilde y^*_n\\m( {\tilde \mu)}-\tilde z^*_n\end{pmatrix}\notag\\
    &-\frac{1}{2}\begin{pmatrix}
        y-y^*_n\\m( \mu)-z^*_n
    \end{pmatrix}^\top X^*_n\begin{pmatrix}
        y-y^*_n\\m( \mu)-z^*_n
    \end{pmatrix}
        -\frac{1}{2}\begin{pmatrix}
         \tilde y-\tilde y^*_n\\m({\tilde \mu})-\tilde z^*_n
    \end{pmatrix}^\top \tilde X^*_n\begin{pmatrix}
        \tilde y-\tilde y^*_n\\m( {\tilde \mu})-\tilde z^*_n
    \end{pmatrix}
\end{align*}
satisfies
\begin{align*} 
G_n(\theta^{0,n},\tilde \theta^{0,n})&= u_1(t^*_n,y^*_n,\mu^{0,n})-v_1(\tilde t^*_n,\tilde y^*_n,\tilde \mu^{0,n})\\
&\geq u_2^*(t^*_n,y^*_n,z^*_n)-v_{2*}(\tilde t^*_n,\tilde y^*_n,\tilde z^*_n)-C_n.
\end{align*}
Thus, thanks to the maximality condition \eqref{maxu2v2} and the definitions of $u_2-v_2$, we have
\begin{align*} 
G_n(\theta^{0,n},\tilde \theta^{0,n})&= u_1(t^*_n,y^*_n,\mu^{0,n})-v_1(\tilde t^*_n,\tilde y^*_n,\tilde \mu^{0,n})\\
&\ge\sup_{\theta,\tilde \theta\in \Theta}G_n(\theta,\tilde \theta)-C_n,
\end{align*}
and we can apply the variational principle in Lemma \ref{lem:var} to get the existence of
$$(\theta^{j,n},\tilde \theta^{j,n}):=((t^{j,n},y^{j,n}\mu^{j,n}),(\tilde t^{j,n},\tilde y^{j,n},\tilde \mu^{j,n}))\mbox{ for }j\geq1$$ and $(\theta^{*,n},\tilde \theta^{*,n}):=((t^{*,n},y^{*,n}\mu^{*,n}),(\tilde t^{*,n},\tilde y^{*,n},\tilde \mu^{*,n}))$
so that  
\begin{align}\label{eq:starn1}
    &d^2_{\sigma}(\theta^{*,n},\theta^{j,n})+d^2_{\sigma}(\tilde \theta^{*,n},\tilde \theta^{j,n})\leq  \frac{2\sqrt {C_n}}{ 2^j}\mbox{ for }j\geq 0\\
    &G_n(\theta^{0,n},\tilde \theta^{0,n})\leq G_n(\theta^{*,n},\tilde \theta^{*,n})-\sqrt {C_n} ( \phi_{2,n}(\theta^{*,n})+\tilde \phi_{2,n}(\tilde\theta^{*,n}))\mbox{ and}\notag\\
    &\label{eq:starn3}
  (\theta,\tilde \theta)\mapsto G_n(\theta,\tilde \theta)-\sqrt {C_n} ( \phi_{2,n}(\theta)+\tilde \phi_{2,n}(\tilde\theta))
\end{align}
admits a strict local maximum at $(\theta^{*,n},\tilde \theta^{*,n})$ with 
\begin{align*}
    \phi_{2,n}(\theta)&:=\sum_{j\geq 0}\frac{d_{\sigma}^2(\theta,\theta^{j,n})}{2^j},\\
    \tilde \phi_{2,n}(\theta)&:=\sum_{j\geq 0}\frac{d_{\sigma}^2(\theta,\tilde \theta^{j,n})}{2^j}.
\end{align*}
Denoting
\begin{align*}
    \Phi_n(\theta):=&\frac{\a\Psi(\mu)}{2}+\sqrt {C_n} \phi_{2,n} (\theta)\\
    &+q_n^*(t-t^*_n)+(p^*_n)^\top \begin{pmatrix}y-y^*_n\\ m( \mu)-z^*_n\end{pmatrix}
   +\frac{1}{2}\begin{pmatrix}
        y-y^*_n\\ m( \mu)-z^*_n
    \end{pmatrix}^\top X^*_n\begin{pmatrix}
        y-y^*_n\\ m( \mu)-z^*_n
    \end{pmatrix}
       \\
        \tilde \Phi_n(\theta):=&\frac{\a\tilde \Psi(\mu)}{2}+\sqrt {C_n} \tilde \phi_{2,n} ( \theta)\\
    &+\tilde q_n^*( t-\tilde t^*_n)
   +(\tilde p^*_n)^\top\begin{pmatrix} y-\tilde y^*_n\\  m( \mu)-\tilde z^*_n\end{pmatrix}+\frac{1}{2}\begin{pmatrix}
         y-\tilde y^*_n\\  m( \mu)-\tilde z^*_n
    \end{pmatrix}^\top \tilde X^*_n\begin{pmatrix}
        y-\tilde y^*_n\\  m( \mu)-\tilde z^*_n
    \end{pmatrix}
\end{align*}
we obtain that
$$u(\theta)-v(\tilde \theta)-  \Phi_n(\theta)- \tilde \Phi_n(\tilde \theta)$$
admits a strict maximum at $(\theta^{*,n},\tilde \theta^{*,n})=((t^{*,n},y^{*,n},\mu^{*,n}),(\tilde t^{*,n},\tilde y^{*,n},\tilde \mu^{*,n}))$.

\vspace{4pt}
{\it Step 4: Taking the limit in $n$.}
Thanks to \eqref{eq:0ndef} and \eqref{eq:tstarn}, we have
\begin{align*}
d_F(\theta^{0,n},\theta^*)&=d_{\sigma}(\theta^{0,n},\theta^*)=\sqrt{|t^*-t^*_n|^2+|y^*-y^*_n|^2+|z^*_n-m(\mu^*)|^2} \to 0\\
d_F(\tilde \theta^{0,n},\tilde \theta^*)&=d_{\sigma}(\tilde \theta^{0,n},\tilde \theta^*)=\sqrt{|\tilde t^*-\tilde t^*_n|^2+|\tilde y^*-\tilde y^*_n|^2+|\tilde z^*_n-m(\tilde \mu^*)|^2} \to 0.
\end{align*} 
This limit combined with \eqref{eq:starn1} for $j=0$ implies that 
\begin{align}\label{eq:cvnstar}
    d_{\sigma}(\theta^{*,n},\theta^*)+d_{\sigma}(\tilde \theta^{*,n},\tilde \theta^*)\to 0.
\end{align}
Due to \cite[Lemma 4.2]{cosso2021master}, \eqref{eq:cvnstar} implies that $(\mu^{*,n},\tilde \mu^{*,n})\to (\mu^{*},\tilde \mu^{*}) $ in the Wasserstein-$2$ distance. \cite[Theorem 6.9]{villani2009optimal} means that the first two moments of $(\mu^{*,n},\tilde \mu^{*,n})$ converge to the first two moments of $(\mu^{*},\tilde \mu^{*})$. As such, these second moments are bounded. 

By direct computation, 
and the definition of $d_\sigma$, 
we have 
\begin{align*}
    |D_y\phi_{2,n}(\theta^{*,n})|&\leq C \sum_{j}\frac{d_\sigma(\theta^{*,n},\theta^{j,n})}{2^j},\,|\pa_t \phi_{2,n}(\theta^{*,n})|\leq C \sum_{j}\frac{d_\sigma(\theta^{*,n},\theta^{j,n})}{2^j},\\
        |D_\mu\phi_{2,n}(\theta^{*,n},x)|&\leq \sum_{j\geq 0}\frac{|m(\mu^{*,n}) - m(\mu^{j,n})|}{2^j}\\
        &+C_d \left( \sigma+\frac{1}{\sigma}(|x|^2+|m(\mu^{*,n})|^2)+\frac{1}{\sigma} \int |x|^2 \, \Sc_0(\mu^{*,n})(dx)\right)\\
        &\leq C_{d,\sigma} \left( C_n^{1/4}+1+|x|^2\right)\\
    |D_{x \mu}\phi_{2,n}(\theta^{*,n},x)|&\leq C_d \left(1+\frac{1}{\sigma^2}(|x|^2+|m(\mu^{*,n})|^2) \right) \leq  C_{d,\sigma} \left( 1+|x|^2\right)\\
        \Hc\phi_{2,n}(\theta) &= \sum_{j\geq 0}\frac{\Hc \rho_\sigma^2(\mu,\mu^{j,n})}{2^j}=4 I_d=D^2_{yy}\phi_{2,n}(\theta),\,D^2_{y\mu}\phi_{2,n}(\theta,x)=0. 
\end{align*}

Thus, as $n\to \infty$, all derivatives of $$\sqrt {C_n} ( \phi_{2,n}(\theta)+\tilde \phi_{2,n}(\tilde\theta))$$ evaluated at $(\theta^{*,n},\tilde \theta^{*,n})$ goes to $0$ in their respective metrics. 
Additionally, thanks to Lemma \ref{lem:propsliced}, Lemma \ref{lem:L}, and \eqref{eq:cvnstar}, we have that 
\begin{align*}
    &D_\mu\Psi(\mu^{*,n})\to D_\mu\Psi(\mu^{*}),\,
    D_{x \mu}\Psi(\mu^{*,n})\to D_{x \mu}\Psi(\mu^{*})\\
    &D_\mu\tilde \Psi(\tilde \mu^{*,n})\to D_\mu\tilde \Psi(\tilde \mu^{*}),\,
    D_{x \mu}\tilde \Psi(\tilde \mu^{*,n})\to D_{x \mu}\tilde \Psi(\tilde \mu^{*})
\end{align*}
Thus, sending $n\to \infty$ concludes the proof.



 \end{proof}

\section{Proof of the comparison result}\label{s.proofcomp}
The proof is divided into several steps. Firstly, by a standard argument for the purpose of contraposition, without loss of generality, we can assume certain monotonicity of the Hamiltonian in the value of function. Then in Step 1, we apply the doubling variable technique, where the variational principle is applied to obtain maximum. In Step 2, we estimate the distance between optimizers using the Lipschitz property of sub/super solutions and the regularity of the hamiltonian. In Step 3, we compare the jets obtained from Theorem~\ref{ishii} to derive a contradiction. 
\begin{proof}[Proof of Theorem \ref{thm:comp}]
    The proof is based on adapting the ideas of \cite[Theorem 3.59]{fabbri2017stochastic}, \cite[Theorem 2]{lions1989viscosity3} to the Wasserstein space, which is a metric space and not a Banach space. Similarly to \cite[Remark 3.9]{ekren2014viscosity}, by passing from $U$ to $e^{2 L_h t} U$, without loss of generality we can assume that for $u\geq  v$, it holds that 
    \begin{align}\label{eq:monou}
    &L_h(u-v) \leq h(\theta,v,p,X^{11} ,f,g,X^{12} ,X^{22}) - h(\theta,u,p,X^{11} ,f,g,X^{12} ,X^{22})
    \end{align}
and $h$ satisfies \eqref{eq:unifh1} with $3L_h$. 
Denote $\chi(\theta):=e^{-\tilde Lt}\left(1+|y|^2+\left(\int |x|^2\mu(dx)\right)^2\right)$ so that 
\begin{align*}
    \pa_t \chi(\theta)&=-\tilde L \chi(\theta),\, D_y \chi(\theta)=2e^{-\tilde Lt}y,\,D^2_{yy} \chi(\theta)=2e^{-\tilde Lt}I_d\\
    D_\mu \chi(\theta,x)&=4e^{-\tilde Lt}\int |z|^2\mu(dz)x,\,D_{x \mu} \chi(\theta,x)=4e^{-\tilde Lt}\int |z|^2\mu(dz)I_d,\,\Dc^2_{y\mu} \chi(\theta)=0\\
    \Hc \chi(\theta)&=e^{-\tilde Lt}\left(8m(\mu)m^\top(\mu)+4\int |x|^2\mu(dx) I_d\right).
\end{align*}
    
The conclusion of the theorem being $\sup_{\theta\in \Theta}(u-v)(\theta)=0$, to obtain a contradiction, we assume that 
$$\sup_{\theta\in \Theta}(u-v)(\theta)>0.$$
Thus, there exists $\iota>0$ so that there exists $r>0$ such that 
$$\sup_{\theta\in \Theta}(u_\iota-v)(\theta)=:2r>0$$
    where $u_\iota(\theta):=u(\theta)-\iota \chi(\theta).$
 We now show there exists $\tilde L$ large enough so that $u_\iota$ is a viscosity subsolution of the same equation. Let $\phi$ be a partially $C^2$ regular test function so that
$u_\iota-\phi$ has a local maximum at $\theta\in \Theta$. This is equivalent to $u-\phi_\iota$ has a local maximum at the same point where $\phi_\iota=\phi+\iota\chi$.
Thus, the viscosity property of $u$ yields 
\begin{align*}
   -\pa_t \phi_\iota(\theta)-h(\theta,u(\theta),D_y \phi_\iota(\theta),D^2_{yy}\phi_\iota(\theta) ,D_\mu \phi_\iota(\theta),D_{x\mu}\phi_\iota(\theta),\Dc^2_{y\mu}\phi_\iota(\theta) ,\Hc \phi_\iota(\theta))\leq 0.
\end{align*}

Thanks to the equality $\phi=\phi_\iota -\iota \chi$ and \eqref{eq:monou},
we can estimate 
\begin{align*}
   &-\pa_t \phi(\theta)-h(\theta,u_\iota(\theta),D_y \phi(\theta),D^2_{yy}\phi(\theta) ,D_\mu \phi(\theta),D_{x\mu}\phi(\theta),\Dc^2_{y\mu}\phi(\theta) ,\Hc \phi(\theta))\\
   &\leq -\pa_t \phi_\iota(\theta)-\iota (\tilde L+L_h)\chi(\theta)\\
   &\quad-h(\theta,u(\theta),D_y  \phi (\theta),D^2_{yy} \phi (\theta) ,D_\mu  \phi (\theta),D_{x\mu} \phi (\theta),\Dc^2_{y\mu} \phi (\theta) ,\Hc  \phi (\theta))\\
   &\leq -\pa_t \phi_\iota(\theta)-h(\theta,u(\theta),D_y \phi_\iota(\theta),D^2_{yy}\phi_\iota(\theta) ,D_\mu \phi_\iota(\theta),D_{x\mu}\phi_\iota(\theta),\Dc^2_{y\mu}\phi_\iota(\theta) ,\Hc \phi_\iota(\theta))\\
   &\quad+3L_h\iota \left(|D_y \chi(\theta)|+|D^2_{yy}\chi(\theta)|+||D_\mu \chi(\theta)||_\mathfrak{q}+||D_{x\mu}\chi(\theta)||_\mathfrak{q}+|\Dc^2_{y\mu}\chi(\theta)|+|\Hc \chi(\theta)|\right)\\
   &\quad-\iota (\tilde L+L_h)\chi(\theta)\\
   &\leq 3L_h\iota \left(|D_y \chi(\theta)|+|D^2_{yy}\chi(\theta)|+||D_\mu \chi(\theta)||_\mathfrak{q}+||D_{x\mu}\chi(\theta)||_\mathfrak{q}+|\Dc^2_{y\mu}\chi(\theta)|+|\Hc \chi(\theta)|\right)\\
   &\quad-\iota (\tilde L+L_h)\chi(\theta)
\end{align*}
where we used the viscosity property of $u$ to obtain the last inequality. 
Using the derivatives of $\chi$, the definition of $\lVert \cdot \rVert_{\mathfrak{q}}$ and Young's inequality we obtain 
\begin{align*}
   &-\pa_t \phi(\theta)-h(\theta,u_\iota(\theta),D_y \phi(\theta),D^2_{yy}\phi(\theta) ,D_\mu \phi(\theta),D_{x\mu}\phi(\theta),\Dc^2_{y\mu}\phi(\theta) ,\Hc \phi(\theta))\\
   &\leq e^{-\tilde Lt}\iota 6L_h \left(|y|+|I_d|+2\int |x|^2\mu(dx)\left(\sup_{x}\frac{|x|}{1+|x|^2}+2|I_d|\right)+4|m(\mu)|^2\right)\\
   &\quad-\iota e^{-\tilde Lt}(\tilde L+L_h)\left(1+|y|^2+\left(\int |x|^2\mu(dx)\right)^2\right)\\
   &\leq e^{-\tilde Lt}\iota  \left(6L_h\sqrt{d}+6L_h|y|+12L_h(3+2\sqrt{d})\int |x|^2\mu(dx)\right)\\
   &\quad-\iota e^{-\tilde Lt}(\tilde L+L_h)\left(1+|y|^2+\left(\int |x|^2\mu(dx)\right)^2\right)\\
   &\leq e^{-\tilde Lt}\iota  \left(6L_h\sqrt{d}+18L_h+\frac{L_h|y|^2}{2}+\frac{(12(3+2\sqrt{d}))^2}{2}L_h+\frac{L_h}{2}\left(\int |x|^2\mu(dx)\right)^2\right)\\
   &\quad-\iota e^{-\tilde Lt}(\tilde L+L_h)\left(1+|y|^2+\left(\int |x|^2\mu(dx)\right)^2\right)\\
   &\leq e^{-\tilde Lt}\iota  \left(6L_h\sqrt{d}+18L_h+\frac{(12(3+2\sqrt{d}))^2}{2}L_h-\tilde L+L_h\right)
\end{align*}
for any choice of $\tilde L\geq 0$. Thus, taking $\tilde L$ large enough, $u_\iota$ is a viscosity subsolution of 
\begin{align*}
   &-\pa_t u_\iota(\theta)-h(\theta,u_\iota(\theta),D_y u_\iota(\theta),D^2_{yy}u_\iota(\theta) ,D_\mu u_\iota(\theta),D_{x\mu}u_\iota(\theta),\Dc^2_{y\mu}u_\iota(\theta) ,\Hc_\iota(\theta))\leq 0
\end{align*}
which also satisfies 
$$u_\iota(T,\cdot)\leq u(T,\cdot)\leq v(T,\cdot).$$
{\it Step 1: Introducing doubling of variable with $d_F$:} We fix $\e>0$, $\iota\in [0,\iota_0]$ and denote $$G_\iota(\theta,\tilde \theta)= u_\iota(\theta)-v(\tilde \theta)-\frac{d^2_F(\theta,\tilde \theta)}{2\e}$$  and introduce the optimization problem 
    \begin{align}
\sup_{(\theta,\tilde \theta)\in\Theta^2}G_\iota(\theta,\tilde \theta).\label{eq:opt1}
    \end{align}
Due to the boundedness from above of $u,-v$, \eqref{eq:opt1} is finite. Fix $\kappa>0$, since 
$$2r=\sup_{\theta\in \Theta}(u_\iota-v)(\theta)\leq \sup_{(\theta, \tilde \theta)\in \Theta^2}G_\iota(\theta, \tilde \theta),$$
we can find $(\theta^0, \tilde \theta^0)\in \Theta^2$ so that 
\begin{align}\label{eq:deftheta0}
    \sup_{(\theta, \tilde \theta)\in \Theta^2}G_\iota(\theta, \tilde \theta)-\kappa<G_\iota(\theta^0, \tilde \theta^0)\mbox{ and }r\leq G_\iota(\theta^0, \tilde \theta^0).
\end{align}
 Using the variational principle in Lemma \ref{lem:var} with the gauge-type function $d_\sigma$, for $\delta=\sqrt{\kappa}$, there exists 
 $$(\theta^*, \tilde \theta^*)=((t^*,y^*,\mu^*),(\tilde t^*,\tilde y^*,\tilde \mu^*))$$ 
 and 
$$\{(\theta^j, \tilde \theta^j):j\geq 1\}=\{((t^j,y^j,\mu^j),(\tilde t^j,\tilde y^j,\tilde \mu^j)):j\geq 1\}$$ 
depending on $\e,\kappa$ such that denoting 
$$\phi(\theta):=\sum_{j=0}^\infty\frac{d_\sigma^2(\theta,\theta^j)}{2^j}\mbox{ and }\tilde \phi(\tilde \theta):=\sum_{j=0}^\infty\frac{d_\sigma^2(\tilde \theta,\tilde \theta^j)}{2^j}$$
we have
 \begin{align}\label{eq:bound1}
    d_\sigma^2(\theta^*,\theta^j)+d_\sigma^2(\tilde \theta^*,\tilde \theta^j)&\leq \frac{\sqrt{\kappa}}{ 2^j}\mbox{ for }j\geq 0    \\
        G_\iota(\theta^0, \tilde \theta^0)&\leq G_\iota(\theta^*, \tilde \theta^*)- \sqrt{\kappa}(\phi(\theta^*)+\tilde \phi(\tilde \theta^*))  \label{eq:boundt} \\
\label{eq:bound2}\mbox{and }
(\theta, \tilde \theta)\in \Theta^2&\mapsto G_\iota(\theta, \tilde \theta)- \sqrt{\kappa}(\phi(\theta)+\tilde \phi(\tilde \theta))  
\end{align}
admits a strict global maximum at $(\theta^*, \tilde \theta^*).$
Note that by a direct computation we have that 
\begin{align}\label{eq:derphi}
    &D_y\phi(\theta)=2\sum_{j=0}^\infty\frac{y-y^j}{2^j},\,D_y\tilde \phi(\tilde \theta)=2\sum_{j=0}^\infty\frac{\tilde y-\tilde y^j}{2^j},\\
    &\mathcal{D}_{y\mu}\phi(\theta)=\mathcal{D}_{y\mu}\tilde \phi(\tilde \theta)=0,\,D^2_{yy}\phi(\theta)=D^2_{yy}\tilde\phi(\tilde \theta)=4I_d,\,\Hc\phi(\theta)=\Hc\tilde \phi(\tilde \theta)=4I_d.\label{eq:derphi2}
\end{align}

{\it Step 2: A priori estimates on the optimizers:} Define the modulus of continuity of $v$ in time defined by 
$$\omega_{v}(t):=\sup\{|v(s,y,\mu)-v(\tilde s,y,\mu)|:|s-\tilde s|\leq t,y\in \R^d,\mu\in \Pc_2(\R^d)\}$$
which satisfies for all $t,s\geq 0$, the inequality
$$\omega_{v}(t+s)\leq \omega_{v}(s)+\omega_{v}(t).$$
Also denote $L_{v,F}$ the Lipschitz constant of $v$ in $(y,\mu)$ w.r.t. $d_F$ metric. 
For $\e,\kappa>0$, define $t_{\e,\kappa}$ the maximum positive root of the equation 
$$t_{\e,\kappa}^2 =2\kappa+2 \omega_{v}(\sqrt{\e}t_{\e,\kappa}).$$ 
Similar to \cite[Theorem]{lasry1986remark}, thanks to the subadditivity of $\omega_{v}$,
we have
$$t_{\e,\kappa}\to 0\mbox{ as }\e,\kappa\downarrow 0.$$
We now show that for all $\e,\kappa>0$, we have
\begin{align}\label{eq:bound:thetas}
   \frac{\sqrt{|y^*-\tilde y^*|^2+\rho^2_F(\mu^*,\tilde \mu^*)}}{\e}\leq L_{v,F}+\sqrt{2 L_{v,F}^2+2\frac{\kappa}{\e}}\mbox{ and } \frac{|t^*-\tilde t^*|}{\sqrt{\e}}\leq t_{\e,\kappa}.
\end{align}
By definition of $G$ and \eqref{eq:deftheta0}-\eqref{eq:bound2} we have the inequalities
\begin{align*}
    &\kappa+u(\theta^*)-\iota \chi(\theta^*)-v(\tilde \theta^*)-\frac{d_F^2(\theta^*,\tilde \theta^*)}{2\e}=\kappa+G_\iota(\theta^*,\tilde \theta^*)\\
    &\geq \kappa+G_\iota(\theta^0, \tilde \theta^0)+\sqrt{\kappa}(\phi(\theta^*)+\tilde \phi(\tilde \theta^*))    \geq  \sup_{(\theta, \tilde \theta)\in \Theta^2}G_\iota(\theta, \tilde \theta)+\sqrt{\kappa}(\phi(\theta^*)+\tilde \phi(\tilde \theta^*))   \\
    &\geq  G_\iota(\theta^*,(t^*,\tilde y^*,\tilde \mu^*))+\sqrt{\kappa}(\phi(\theta^*)+\tilde \phi(\tilde \theta^*))  \\
    &\geq  u(\theta^*)-\iota \chi(\theta^*)-v(t^*,\tilde y^*,\tilde \mu^*)+\sqrt{\kappa}(\phi(\theta^*)+\tilde \phi(\tilde \theta^*))   .
\end{align*}
Thus, given the Lipschitz continuity of $v$, we obtain the inequality
\begin{align*}
    \frac{d_F^2(\theta^*,\tilde \theta^*)}{2\e} +\sqrt{\kappa}(\phi(\theta^*)+\tilde \phi(\tilde \theta^*))&\leq \kappa+v( \theta^*)-v(t^*,\tilde y^*,\tilde \mu^*) \\
    &\leq \kappa+ L_{v,F}\sqrt{|y^*-\tilde y^*|^2+\rho^2_F(\mu^*,\tilde \mu^*)}\notag
\end{align*}
This inequality implies that  $b=\frac{\sqrt{|y^*-\tilde y^*|^2+\rho^2_F(\mu^*,\tilde \mu^*)}}{\e}$ satisfies 
\begin{align*}
     \left(b-L_{v,F}\right)^2\leq 2 L_{v,F}^2+2\frac{\kappa}{\e} 
\end{align*}
which implies that 
$$\frac{\sqrt{|y^*-\tilde y^*|^2+\rho^2_F(\mu^*,\tilde \mu^*)}}{\e}\leq L_{v,F}+\sqrt{2 L_{v,F}^2+2\frac{\kappa}{\e}}.$$
Similarly, \eqref{eq:deftheta0}-\eqref{eq:bound2} lead to 
\begin{align*}
    \frac{d_F^2(\theta^*,\tilde \theta^*)}{2\e} +\sqrt{\kappa}(\phi(\theta^*)+\tilde \phi(\tilde \theta^*))&\leq \kappa+v( \theta^*)-v(\tilde t^*, y^*, \mu^*) \\
    &\leq \kappa+ \omega_{v}(|t^*-\tilde t^*|).
\end{align*}
Thus, 
$a=\frac{|t^*-\tilde t^*|}{\sqrt{\e}}$
satisfies 
$$a^2\leq 2\kappa +2\omega_v(\sqrt{\e}a)$$
which shows that  
\begin{align*}
   \frac{|t^*-\tilde t^*|}{\sqrt{\e}}\leq t_{\e,\kappa}
\end{align*}
since $t_{\e,\kappa}$ is the largest root of this equation. 
We also define $t_\e$ as the largest positive root of
$$t_\e^2=2\omega_{v}(\sqrt{\e}t_\e).$$
Note that $t_\e>0$ if $\omega_v$ is not a constant function. Indeed, if $t_\e=0$, then for all $t>0$, we have 
$$t^2\geq 2\omega_{v}(\sqrt{\e}t).$$
By subadditivity, this implies that 
$$0\leq \frac{\omega_{v}(\sqrt{\e}(t+s))-\omega_{v}(\sqrt{\e}s)}{t}\leq \frac{\omega_{v}(\sqrt{\e}t)-\omega_{v}(0)}{t}\leq \frac{t}{2}.$$
Sending $t\to 0$ we have that $\omega_v$ is a  constant function. 

For the remaining of the proof, we assume that $\omega_v$ is not constant and for each $\e>0$, we choose $\kappa_\e$ small enough so that $t_{\e,\kappa}\leq 2 t_\e$ and $\kappa_\e \leq \e$ which implies that for all $\e>0$ and  $\kappa\in (0,\kappa_\e]$
we have 
\begin{align}\label{eq:bound:thetas2}
   \frac{\sqrt{|y^*-\tilde y^*|^2+\rho^2_F(\mu^*,\tilde \mu^*)}}{\e}\leq L_{v,F}+\sqrt{2 (L_{v,F}^2+2)}\mbox{ and } \frac{|t^*-\tilde t^*|}{\sqrt{\e}}\leq 2t_\e.
\end{align}

We use \eqref{eq:deftheta0}-\eqref{eq:bound2} one more time to obtain
\begin{align*}
    &\kappa+u(\theta^*)-\iota \chi(\theta^*)-v(\tilde \theta^*)-\frac{d_F^2(\theta^*,\tilde \theta^*)}{2\e}  \geq  G_\iota(( t^*,0,\delta_0),\tilde \theta^*)+\sqrt{\kappa}(\phi(\theta^*)+\tilde \phi(\tilde \theta^*))  \\
    &\geq  u( t^*,0,\delta_0)-\iota \chi( t^*,0,\delta_0)-v(\tilde t^*,\tilde y^*,\tilde \mu^*)+\sqrt{\kappa}(\phi(\theta^*)+\tilde \phi(\tilde \theta^*))   .
\end{align*}
Thus, we obtain a prior bound on $\theta^*$,
\begin{align}\label{ea:boundthetastar}
    \frac{\kappa+|u(\theta^*)-u( t^*,0,\delta_0)|}{\iota} 
    &\geq  \chi(\theta^*)-\chi( t^*,0,\delta_0)\\
    &=e^{-\tilde Lt^*}\left(|y^*|^2+\left(\int |x|^2\mu^*(dx)\right)^2\right)\notag  .
\end{align}

{\it Step 3: Using Ishii's Lemma :}
    Define the functions
    \begin{align*}
        \Lc(\mu,\eta,\nu)&:=2\int \frac{Re(F_k(\Sc_0(\mu))(F_k(\Sc_0(\eta))-F_k(\Sc_0(\nu)))^*)}{(1+|k|^2)^\lambda}dk\\
        &+2Tr(V^\top(\Sc_0(\mu))(V(\Sc_0(\eta))-V(\Sc_0(\nu))))\\
        \Psi(\mu)&:=2\rho^2_F(\Sc_0(\mu),\Sc_0(\mu^*))+\Lc(\Sc_0(\mu),\Sc_0(\mu^*),\Sc_0(\tilde \mu^*))\\
        \tilde \Psi(\mu)&:= 2\rho^2_F(\Sc_0(\mu),\Sc_0(\tilde \mu^*))-\Lc(\Sc_0(\mu),\Sc_0(\mu^*),\Sc_0(\tilde \mu^*))\\
         u_1(t,x,\mu)&:= u_\iota(t,x,\mu)-\sqrt{\kappa}\phi(t,x,\mu)  \\
         v_1(t, x,\mu)&:=v(t, x,\mu)+\sqrt{\kappa} \tilde \phi(t,x,\mu)
    \end{align*}
We now show that one can apply the Ishii's Lemma in Theorem \ref{ishii}.
Since $\rho_F$ and $\rho_\sigma$ topologically equivalent, $u_1,-v_1$ are $d_F$ upper semicontinuous in both topology and the property \eqref{cond:unifc0} holds for $u_1,v_1.$ It is now sufficient to show that one can choose $\e>0$ so that $t^*<T$ and $\tilde t^*<T$. Note that $t^*,\tilde t^*$
are bounded. Denote $t,\tilde t$ an accumulation point as $\e\downarrow 0$, $0<\kappa<\kappa_\e$. Thanks to \eqref{eq:bound:thetas2}, we have the equality $t=\tilde t.$
The non-negativity of $\phi,\tilde \phi$, \eqref{eq:deftheta0}, and \eqref{eq:boundt} implies that 
\begin{align}\label{eq:signstar}
    u_\iota(\theta^*)\geq v(\tilde \theta^*)+r.
\end{align}
Thus, thanks to the values of $u_\iota,v$ at $T,$ we have
\begin{align*}
    &\omega_{u_\iota}(|t^*-T|)+\omega_{v}(|\tilde t^*-T|)+L_{v,F}d_F((T,\tilde y^*,\tilde \mu^*),(T, y^*, \mu^*))\\
    &\geq v(T, y^*, \mu^*)-u_\iota(T,y^*,\mu^*)+r\geq r.
\end{align*}
Note that for $0<\kappa<\kappa_\e$ as $\e\downarrow 0$, we have $d_F((T,\tilde y^*,\tilde \mu^*),(T, y^*, \mu^*))\to 0$ thanks to \eqref{eq:bound:thetas2}. Thus, if $\tilde t=t=T$, we obtain $0\geq r$ which would conclude the proof. 
Thus, for the remainder of the proof we assume that $\tilde t=t<T$ and we choose $\e>0$, so that for all $0<\kappa<\kappa_\e$, 
$\max\{t^*,\tilde t^*\}<T$.

We can now use Theorem \ref{ishii} to obtain that 
there exist $X^*,\tilde X^*\in \Sc_{2d}$ so that 
\begin{align}
    &\Big(\frac{t^*-\tilde t^*}{\e},\frac{y^*-\tilde y^*}{\e},(X^*)^{11},\frac{2(m(\mu^*)-m({\tilde \mu}^*))+D_\mu\Psi(\mu^*)}{2\e},\notag\\
    &\quad\quad\quad\quad\quad\quad\quad\quad\quad\quad\quad\quad\quad\quad\frac{D_{x \mu}\Psi(\mu^*)}{2\e},(X^*)^{12},(X^*)^{22}\Big)\in\bar J^{2,+} u_1(\theta^*)\notag\\
    &\Big(\frac{t^*-\tilde t^*}{\e},\frac{y^*-\tilde y^*}{\e},-(\tilde X^*)^{11},\frac{2(m(\mu^*)-m({\tilde \mu}^*))-D_\mu\tilde \Psi(\tilde \mu^*)}{2\e},\notag\\
    &\quad\quad\quad\quad\quad\quad\quad\quad\quad\quad\quad\quad-\frac{D_{x\mu}\tilde \Psi(\tilde \mu^*)}{2\e},-(\tilde X^*)^{12} ,- (\tilde X^*)^{22}\Big)\in \bar J^{2,-} v_1(\tilde \theta^*)\notag\\
\label{eq:xstar}
               \mbox{and } &-\frac{3}{\e}I_{4d}\leq 
\begin{pmatrix}
    X^*&0\\0&\tilde X^*
\end{pmatrix}
\leq 
\frac{3}{\e}
\begin{pmatrix}
    I_{2d}&-I_{2d}\\-I_{2d}&I_{2d}
\end{pmatrix}.
\end{align}

Given the smoothness of $\phi, \tilde \phi$, the viscosity properties of the function $u_\iota,v$ and the continuity of $h$ leads to 
 \begin{align*}
    0&\leq h\left(\theta^*,u_\iota( \theta^*),\frac{y^*-\tilde y^*}{\e}+\sqrt{\kappa} D_y\phi(\theta^*),(X^*)^{11}+4\sqrt{\kappa} I_d,\right.\\
    &\quad\quad\quad\left.\frac{m(\mu^*)-m( {\tilde \mu}^*)}{\e}+\frac{D_\mu \Psi( \mu^*)}{2\e}+\sqrt{\kappa} D_\mu \phi( \theta^*),\right.\\
    &\quad\quad\quad\left.\frac{D_{x\mu}\Psi(\mu^*)}{2\e}+\sqrt{\kappa} D_{x\mu}\phi(\theta^*),(X^*)^{12},(X^*)^{22}+4\sqrt{\kappa} I_d\right)\\
     &-h\left(\tilde \theta^*,v(\tilde \theta^*),\frac{y^*-\tilde y^*}{\e}-\sqrt{\kappa} D_y\tilde \phi(\tilde \theta^*),-(\tilde X^*)^{11}-4\sqrt{\kappa} I_d,\right.\\
    &\quad\quad\quad\left.\frac{m(\mu^*)-m( {\tilde \mu}^*)}{\e}-\frac{D_\mu\tilde \Psi(\tilde \mu^*)}{2\e}-\sqrt{\kappa} D_\mu\tilde \phi(\tilde \theta^*),\right.\\
     &\quad\quad\quad\left.-\frac{D_{x\mu}\tilde \Psi(\tilde \mu^*)}{2\e}-\sqrt{\kappa} D_{x\mu}\tilde \phi(\tilde \theta^*),-(\tilde X^*)^{12} ,-(\tilde X^*)^{22}-4\sqrt{\kappa} I_d\right).
 \end{align*}
 Thus, thanks to \eqref{eq:signstar} and \eqref{eq:monou}, we have
 \begin{align*}
    &L_h r\leq L_h(u_\iota(\theta^*)-v(\tilde \theta^*))\\
    &\leq h\left(\theta^*,v(\tilde \theta^*),\frac{y^*-\tilde y^*}{\e}+\sqrt{\kappa} D_y\phi(\theta^*),(X^*)^{11}+4\sqrt{\kappa} I_d,\right.\\
    &\quad\quad\quad\left.\frac{m(\mu^*)-m( {\tilde \mu}^*)}{\e}+\frac{D_\mu \Psi( \mu^*)}{2\e}+\sqrt{\kappa} D_\mu \phi( \theta^*),\right.\\
    &\quad\quad\quad\left.\frac{D_{x\mu}\Psi(\mu^*)}{2\e}+\sqrt{\kappa} D_{x\mu}\phi(\theta^*),(X^*)^{12},(X^*)^{22}+4\sqrt{\kappa} I_d\right)\\
     &-h\left(\tilde \theta^*,v(\tilde \theta^*),\frac{y^*-\tilde y^*}{\e}-\sqrt{\kappa} D_y\tilde \phi(\tilde \theta^*),-(\tilde X^*)^{11}-4\sqrt{\kappa} I_d,\right.\\
    &\quad\quad\quad\left.\frac{m(\mu^*)-m( {\tilde \mu}^*)}{\e}-\frac{D_\mu\tilde \Psi(\tilde \mu^*)}{2\e}-\sqrt{\kappa} D_\mu\tilde \phi(\tilde \theta^*),\right.\\
     &\quad\quad\quad\left.-\frac{D_{x\mu}\tilde \Psi(\tilde \mu^*)}{2\e}-\sqrt{\kappa} D_{x\mu}\tilde \phi(\tilde \theta^*),-(\tilde X^*)^{12} ,-(\tilde X^*)^{22}-4\sqrt{\kappa} I_d\right).
 \end{align*}
For fixed $\e>0$, the family 
$$\left(\frac{\begin{pmatrix}
    y^*-\tilde y^*\\ m (\mu^*)-m( {\tilde {\mu}}^*)
\end{pmatrix}}{\e},X^*,\tilde X^*\right)=\left(\frac{\begin{pmatrix}
    y^*-\tilde y^*\\ m(\mu^*)-m({\tilde {\mu}}^*)
\end{pmatrix}}{\e},X^*,\tilde X^*\right)_{\kappa}$$
indexed by $\{\kappa:\,\kappa_\e\geq \kappa>0\}$ is bounded thanks to \eqref{eq:bound:thetas2}, \eqref{eq:xstar} and the inequality $t_{\e,\kappa}\leq 2t_\e$. Similarly to above, for fixed $\e>0$, $\{t^*,\tilde t^*\}$ is also bounded and as $\kappa\downarrow 0$ admits an accumulation point. Let $(t_{1,\e},\tilde t_{1,\e},p,X,\tilde X)$ be an accumulation point of this family as $\kappa\downarrow 0$. Taking $\e>0$ small enough, we can insure that $t_{1,\e}<T$, $\tilde t_{1,\e}<T$, and $(p,X,\tilde X)$ satisfies 
\begin{align}\label{eq:orderXn2}
        &-\frac{3}{\e}I_{4d}\leq 
\begin{pmatrix}
    X&0\\0&\tilde X
\end{pmatrix}
\leq 
\frac{3}{\e}
\begin{pmatrix}
    I_{2d}&-I_{2d}\\-I_{2d}&I_{2d}
\end{pmatrix}\mbox{ and }|p|\leq C
\end{align}
thanks to \eqref{eq:bound:thetas2}.

Additionally, thanks to \eqref{eq:bound1} applied for $j=0$ and \eqref{eq:bound:thetas2} we have that for $\kappa>0$ small enough
$$d_\sigma^2(\theta^*,\theta^0)+d_\sigma^2(\tilde \theta^*,\tilde \theta^0)+d_F(\theta^*,\tilde \theta^*)\leq 2t_\e\sqrt{\e}.$$
Denote by $p_1,p_2$ the first and second $d$ coordicates of $p$, and 
\begin{align*}
    \Delta:=&\left|p_1-\frac{y^*-\tilde y^*}{\e}\right|+\left|p_2-\frac{m(\mu^*)-m(\tilde \mu^*)}{\e}\right|+|X-X^*|\\
    &+ \sqrt{\kappa}\left( |D_y\phi(\theta^*)|+||D_\mu\phi(\theta^*,\cdot)||_q+||D_{x\mu}\phi(\theta^*,\cdot)||_q+4|I_d|\right)\\
    \tilde \Delta:=&\left|p_1-\frac{y^*-\tilde y^*}{\e}\right|+\left|p_2-\frac{m(\mu^*)-m(\tilde \mu^*)}{\e}\right|+|\tilde X-\tilde X^*|\\
    &+ \sqrt{\kappa}\left( |D_y\tilde \phi(\tilde \theta^*)|+||D_\mu\tilde \phi(\tilde \theta^*,\cdot)||_{\mathfrak{q}}+||D_{x\mu}\tilde \phi(\tilde \theta^*,\cdot)||_{\mathfrak{q}}+4|I_d|\right)\\
    I(x)&:=\frac{D_\mu\Psi(\mu^*,x)+D_\mu\tilde \Psi(\tilde \mu^*,x)}{2\e}
\end{align*}
so that by the definition of $p,X,\tilde X$, \eqref{eq:bound1}, and the expresssions of derivatives of $\phi,\tilde \phi$, 
$\Delta+\tilde \Delta\to 0$ as $\kappa\downarrow 0$.
By Lemmas \ref{lem:propsliced}-\ref{lem:L}, 
we have
\begin{align*}
    \frac{D_{\mu}\Psi(\mu,x)}{2\e}&=\frac{2}{\e}\sum_{i,j} (V^{i,j}(\Sc_0(\mu))-V^{i,j}(\Sc_0(\mu^*)))\big((x^j-m^j( \mu)) e_i+(x^i-m^i( \mu)) e_j\big)\\
    &+\frac{2}{\e}\int\frac{Re\left(i k\left(  F_k(\Sc_0(\mu))-f_k(x-m(\mu))\right)(F_k(\Sc_0(\mu))-F_k(\Sc_0(\mu^*)))^*\right)}{(1+|k|^2)^\lambda}dk\\
   &+\frac{1}{\e}\int \frac{Re(i k (F_k(\Sc_0(\mu))-f_k(x-m( \mu)))(F_k( \Sc_0(\mu^*))-F_k(\Sc_0(\tilde \mu^*)))^*)}{(1+|k|^2)^\lambda}dk\\
    &+\frac{1}{\e}\sum_{i,j}\big((x^j-m^j( \mu)) e_i+(x^i-m^i( \mu)) e_j\big)(V^{i,j}( \Sc_0(\mu^*))-V^{i,j}(\Sc_0(\tilde \mu^*)))\\
    \frac{D_{\mu}\tilde \Psi(\tilde \mu,x)}{2\e}&=\frac{2}{\e}\sum_{i,j} (V^{i,j}(\Sc_0(\tilde \mu))-V^{i,j}(\Sc_0(\tilde \mu^*)))\big((x^j-m^j( {\tilde \mu})) e_i+(x^i-m^i( {\tilde \mu})) e_j\big)\\
    &+\frac{2}{\e}\int\frac{Re\left(i k\left(  F_k(\Sc_0(\tilde \mu))-f_k(x-m({\tilde \mu}))\right)(F_k(\Sc_0(\tilde \mu))-F_k(\Sc_0(\tilde \mu^*)))^*\right)}{(1+|k|^2)^\lambda}dk\\
   &-\frac{1}{\e}\int \frac{Re(i k (F_k(\Sc_0(\tilde \mu))-f_k(x-m( {\tilde \mu})))(F_k(  \Sc_0(\mu^*))-F_k(\Sc_0(\tilde \mu^*)))^*)}{(1+|k|^2)^\lambda}dk\\
    &-\frac{1}{\e}\sum_{i,j}\big((x^j-m^j( {\tilde \mu})) e_i+(x^i-m^i({\tilde  \mu})) e_j\big)(V^{i,j}( \Sc_0(\mu^*))-V^{i,j}(\Sc_0(\tilde \mu^*)))
\end{align*}
Thus, given the symmetry of $V$ and the cancellation with the choice $\mu=\mu^*$ and $\mu=\tilde \mu^*$, we obtain
\begin{align}\label{eq:derpsi}
    &\frac{D_{\mu}\Psi(\mu^*,x)}{2\e}=\frac{2}{\e}\Big(V( \Sc_0(\mu^*))-V(\Sc_0(\tilde \mu^*))\Big)(x-m( \mu^*)) \\     &\quad\quad+\frac{1}{\e}\int \frac{Re(i k (F_k(\Sc_0(\mu^*))-f_k(x-m( \mu^*)))(F_k( \Sc_0(\mu^*))-F_k(\Sc_0(\tilde \mu^*)))^*)}{(1+|k|^2)^\lambda}dk\notag\\
    &\notag\frac{D_{\mu}\tilde \Psi(\tilde \mu^*,x)}{2\e}=-\frac{2}{\e}\Big(V( \Sc_0(\mu^*))-V(\Sc_0(\tilde \mu^*))\Big)(x-m( {\tilde \mu^*}))\\
    &\quad\quad-\frac{1}{\e}\int \frac{Re(i k (F_k(\Sc_0(\tilde \mu^*))-f_k(x-m( {\tilde \mu^*})))(F_k(  \Sc_0(\mu^*))-F_k(\Sc_0(\tilde \mu^*)))^*)}{(1+|k|^2)^\lambda}dk\notag
    \end{align}
and
\begin{align*}
    I(x)&=\frac{1}{\e}\int \frac{Re(i k (f_k(x-m( {\tilde \mu^*}))-f_k(x-m( \mu^*)))(F_k( \Sc_0(\mu^*))-F_k(\Sc_0(\tilde \mu^*)))^*)}{(1+|k|^2)^\lambda}dk\notag\\
    &+\frac{2}{\e}\Big(V( \Sc_0(\mu^*))-V(\Sc_0(\tilde \mu^*))\Big)(m( {\tilde \mu^*}- \mu^*)).\notag
\end{align*}
Using $k$-Lipschitz continuity of $x\mapsto f_k(x)$, and the Cauchy Schwarz inequality, we obtain  
\begin{align*}
    |I(x)|\leq &\frac{ |m( {\tilde \mu^*}-\mu^*)|}{\e}\int |k|^2|F_k( \Sc_0(\mu^*))-F_k(\Sc_0(\tilde \mu^*))|\frac{dk}{(1+|k|^2)^\lambda}+\frac{Cd^2_F(\theta^*,\tilde \theta^*)}{\e}\\
    &\leq\frac{d^2_F(\theta^*,\tilde \theta^*)}{\e}\sqrt{\int \frac{dk}{(1+|k|^2)^{\lambda-2}}}+\frac{Cd^2_F(\theta^*,\tilde \theta^*)}{\e}\\
    |D_xI(x)|\leq &\frac{ |m( {\tilde \mu^*}-\mu^*)|}{\e}\int |k|^3|F_k( \Sc_0(\mu^*))-F_k(\Sc_0(\tilde \mu^*))|\frac{dk}{(1+|k|^2)^\lambda}\\
    &\leq\frac{d^2_F(\theta^*,\tilde \theta^*)}{\e}\sqrt{\int \frac{dk}{(1+|k|^2)^{\lambda-3}}}.
\end{align*}
Since $2(\lambda-3)>d$, we obtain from \eqref{eq:bound:thetas2} that 
\begin{align*}
    |I(x)|+ |D_xI(x)|\leq \frac{Cd^2_F(\theta^*,\tilde \theta^*)}{\e}\leq 4Ct_\e^2.
\end{align*}
for $\kappa\in (0,\kappa_\e].$

Thanks to our Lipschitz continuity assumption on $h$ and \eqref{eq:bound:thetas2}-\eqref{ea:boundthetastar}, we have that 
 \begin{align*}
    L_hr&\leq h\left(\theta^*,v(\tilde \theta^*), p_1,( X)^{11}, p_2+\frac{D_\mu\Psi(\mu^*)}{2\e},\frac{D_{x\mu}\Psi(\mu^*)}{2\e},( X)^{12},( X)^{22}\right)\\
     &-h\left(\tilde \theta^*,v(\tilde \theta^*), p_1,-(\tilde X)^{11} ,p_2+\frac{D_\mu\Psi(\mu^*)}{2\e},\frac{D_{x\mu}\Psi(\mu^*)}{2\e},-(\tilde X)^{12} ,-(\tilde X)^{22}\right)\\
     &+3L_h\Delta \left(1+|y^*|+\int |x|^2 \mu^*(dx)\right)\\
     &+3L_h \left(1+|\tilde y^*|+\int |x|^2\tilde \mu^*(dx)\right)(\tilde \Delta+ ||I||_{\mathfrak{q}}+||D_x I||_{\mathfrak{q}})\\
     &\leq h\left(\theta^*,v(\tilde \theta^*), p_1,( X)^{11}, p_2+\frac{D_\mu\Psi(\mu^*)}{2\e},\frac{D_{x\mu}\Psi(\mu^*)}{2\e},( X)^{12},( X)^{22}\right)\\
     &-h\left(\tilde \theta^*,v(\tilde \theta^*), p_1,-(\tilde X)^{11} ,p_2+\frac{D_\mu\Psi(\mu^*)}{2\e},\frac{D_{x\mu}\Psi(\mu^*)}{2\e},-(\tilde X)^{12} ,-(\tilde X)^{22}\right)\\
     &+C\left(1+\e+\sqrt{\frac{\sup|u|+\kappa+1}{\iota}}\right)(\Delta+\tilde \Delta+ 4Ct_\e^2)
\end{align*}
for some constant $C$ that only depends on the dimension and the Lipscthiz constants.

The inequalities \eqref{eq:xstar} and \eqref{eq:contpsi} allow us to obtain there exists a modulus of continuity $w_h$ that might depend on the upper bound in \eqref{eq:bound:thetas2} so that
\begin{align*}
 L_hr&\leq \omega_h\left(d_F(\theta^*,\tilde \theta^*)\right)\left(1+\sqrt{\frac{\sup|u|+\kappa+1}{\iota}}\right)\\
 &+C\left(1+\e+\sqrt{\frac{\sup|u|+\kappa+1}{\iota}}\right)(\Delta+\tilde \Delta+ 4Ct_\e^2)
\\
 &\leq C\left(1+\e+\sqrt{\frac{\sup|u|+\kappa+1}{\iota}}\right)(\omega_h\left(4t_\e^2+2t_\e \sqrt{\e}\right)+\Delta+\tilde \Delta+ 4Ct_\e^2)
 \end{align*}
 for all $\kappa \in (0,\kappa_\e]$.
Taking $\e>0$, $\kappa\in (0,\kappa_\e]$ small enough, we obtain the contradiction with $r>0$.

\end{proof}

\section{Proofs for stochastic control with partial observation}\label{s.proofe}

The following estimate relies on the Sobolev embedding theorem, and the fact that $m^{t,\mu,\alpha}_s$ is a push-forward distribution of some SDE. 

\begin{proof}[Proof of Lemma~\ref{lem:measureflow_estimate}]
Let us take 
$$Y^{t,\alpha}_s=\int_t^s \tilde \sigma(\alpha_u) \, dW_u, \quad \tilde X^{t,\mu,\alpha}_s=X^{t,\mu,\alpha}_s-Y^{t,\alpha}_s, \quad \tilde m^{t,\mu,\alpha}_s=\mathcal{L}(\tilde X^{t,\mu,\alpha}_s \, | \, \mathcal{F}_s^W).$$
Then it can be easily checked that 
\begin{align*} 
d\tilde{X}^{t,\mu,\alpha}_s= b(\tilde{X}^{t,\mu,\alpha}_s+Y^{t,\alpha}_s, \alpha_s) \, ds + \sigma(\tilde{X}^{t,\mu,\alpha}_s+Y^{t,\alpha}_s, \alpha_s) \, dV_s,
\end{align*}
and also $m^{t,\mu,\alpha}_s=(I_d+Y^{t,\alpha}_s)_{\sharp}\tilde m_s$ $a.s.$

Noting that $Y_s^{t,\alpha}$ is known given $\mathcal{F}^W_s$, and hence by standard estimates of SDE regarding $\tilde X_s$, 
\begin{align*}
   W_2(\tilde m_s^{t,\mu,\alpha}, \mu )^2 \leq  L(s-t) \quad a.s.,
\end{align*}
where $L$ is a positive constant depending only on the coefficients $b$, $\sigma$, and $\tilde \sigma$. Making use of the triangle inequality of Wasserstein metric and the fact that $$W_2(m^{t,\mu,\alpha}_s,\tilde m^{t,\mu,\alpha}_s)^2 \leq |Y^{t,\alpha}_s|^2 \quad a.s., $$
we get
$$
W_2(m^{t,\mu,\alpha}_s, \mu)^2 \leq L(s-t + |Y^{t,\alpha}_s|^2) \quad a.s.
$$
Therefore $s \mapsto m^{t,\mu,\alpha}_s(\omega)$ is continuous for almost every $\omega$ and 
$$
\E[W_2(m^{t,\mu,\alpha}_s, \mu)^2] \leq L(s-t). 
$$

Finally, thanks to the Sobolev embedding, we know that any $f \in \mathbb{H}_{\lambda}$ is Lipschitz, and hence $\lVert \nu-\eta\rVert_{-\lambda} \leq C W_1(\nu,\eta)\leq C W_2(\nu,\eta)$ for any $\nu,\eta \in \mathcal{P}_2(\R^d)$, where $C$ is a positive constant from Sobolev embedding. It follows that $\mathbb{E}[\lVert m^{t,\mu,\alpha}_s- \mu \rVert_{-\lambda}^2] \leq CL(s-t)$.
    
\end{proof}

The proof of the dynamic programming principle is standard after showing the Lipschitz property of value function, which heavily depends on the linearity of $\mu \mapsto J(t,\mu,\alpha)$.

\begin{proof}[Proof of Proposition~\ref{prop:Lipschizvalue}]
Note that  $(t,\mu) \mapsto J(t,\mu,\alpha)$ being $L$-Lipschitz uniformly for all admissible controls $\alpha$ implies the $L$-Lipschitz continuity of $(t,\mu) \mapsto v(t,\mu)$. Also as commented in \cite{bandini2019randomized}, with the continuity of value functions the verification of dynamic programming principle is standard. Therefore, it suffices to show the Lipschitz property of $(t,\mu) \mapsto J(t,\mu,\alpha)$. In the rest of the proof, let us fix a control $\alpha$, and suppress $\alpha$ in the notation $X^{t,x,\alpha}_s$, where $X^{t,x,\alpha}_s$ denotes the solution starting from $\mu=\delta_x$.

\vspace{4pt} 
\textit{Step I.}
Denoting $G(x):=\E[g(X_u^{t,x})]$, we prove that $G \in \mathbb{H}_{\lambda}$. Due to our assumption, $g$ is of exponential decay, and hence we have the estimate 
\begin{align*}
  \E[ g(X_u^{t,x})] &\leq K\E\left[e^{-c|X_u^{t,x}|}\right]=K \int_0^{\infty} e^{-cr} \, \mathbb{P}(|X_{u}^{t,x}| \in dr) \\
    &=\frac{K}{c} \int_0^{\infty} \int_0^{\infty} \mathbbm{1}_{\{z \geq r\}}e^{-cz} \, dz \mathbb{P}(|X_{u}^{t,x}| \in dr) \\
    &= \frac{K}{c}\int_0^{\infty} e^{-cz} \mathbb{P}(|X_u^{t,x}| \leq z) \, dz
\end{align*}
where we use Fubini's theorem in the last equality.  For $z \leq |x|/2$, we have that $$\mathbb{P}(|X_u^{t,x}| \leq z) \leq \mathbb{P}(|X_u^{t,x}-x| \geq |x|-z) \leq \frac{\E[|X_u^{t,x}-x|^{\beta}]}{(|x|/2)^{\beta}}.$$
Since $b,\sigma$ are uniformly bounded, we have the boundedness of $\E[|X_u^{t,x}-x|^{\beta}]$ for any $\beta \geq 2$, and hence for $|x| \geq 1$
\begin{align}\label{eq:costintegrability}
    \E[g(X_u^{t,x})]\leq \frac{K}{c} \int_{|x|/2}^{\infty} e^{-cz} \, dz + \frac{K}{c} \int_{0}^{|x|/2} \frac{C}{(|x|/2)^{\beta}} \, dz \leq C\left(e^{-c|x|}+ \frac{1}{|x|^{\beta-1}} \right),
\end{align}
where $C$ is a constant depending on $\alpha$, $\lVert b \rVert_{\infty}$, and $ \lVert \sigma \rVert_{\infty}$. Choosing $\beta \geq d+1$, it shows that $x \mapsto \E[g(X_u^{t,x})]$ is $L^2$-integrable. 

According to \cite[Chapter 3]{MR1472487} $x \mapsto X_u^{t,x}$ is $\lambda$-th differentiable, and we denote its $k$-th order derivative by $X^{t,x,(k)}_u$. Let us now estimate the first-order derivative of $G$. We actually have that $$|G^{(1)}(x)|= \left|\E\left[g^{(1)}(X_u^{t,x}) X_u^{t,x,(1)}\right]\right| \leq \sqrt{\E\left[\left|g^{(1)}(X_u^{t,x})\right|^2 \right]}\sqrt{\E\left[\left|X^{t,x,(1)}_u\right|^2 \right]}. $$
According to our assumptions, $x \mapsto g^{(1)}(x)$ is of exponential decay, and therefore $x \mapsto \sqrt{\E\left[\left|g^{(1)}(X_u^{t,x})\right|^2 \right]}$ is also $L^2$-integrable due to \eqref{eq:costintegrability} for a suitable choice of $\alpha$. It suffices to show that the $L^2$-norm of $X_u^{t,x,(1)}$ is uniformly bounded. According to \cite[Chapter 3]{MR1472487}, $X_u^{t,x,(1)}$ satisfies the SDE 
\begin{align*}
    dX_u^{t,x,(1)}=&b^{(1)}(X_u^{t,x},\alpha_u) X_u^{t,x,(1)} \,du+ \sigma^{(1)}(X_u^{t,x},\alpha_u) X_u^{t,x,(1)} \, dV_u \\
    X_t^{t,x,(1)}=&\mathbf{1},
\end{align*}
where $\mathbf{1}$ represents a $d$-dimensional vector with $1$ at each entry. Since $\lVert b^{(1)} \rVert_{\infty}+\lVert \sigma^{(1)} \rVert_{\infty} <+\infty$, the SDE above is well-posed and has a unique solution. By standard estimate and the BDG inequality, it can be easily seen that for any $\beta \geq 2$
\begin{align}\label{eq:SDEintegrability}
    \E[|X_u^{t,x,(1)}|^{\beta}]\leq C\left(1+ \int_t^u \E[|X_u^{t,x,(1)}|^{\beta}] \,du  \right).
\end{align}
Invoking Gr\"{o}nwall's inequality, we prove that $ \E[|X_u^{t,x,(1)}|^{\beta}]$ is bounded uniformly in $x$ for any fixed $\beta \geq 2$.

Induction: by chain rule we get that 
\begin{align*}
    |G^{(2)}(x)|= \left| \E\left[g^{(2)}(X_u^{t,x}) |X_u^{t,x,(1)}|^2+ g^{(1)}(X_u^{t,x})X_u^{t,x,(2)} \right]\right|. 
\end{align*}
Thanks to \eqref{eq:costintegrability} and \eqref{eq:SDEintegrability} for the suitable choice of $\alpha,\beta$, the function $x \mapsto \E\left[g^{(2)}(X_u^{t,x}) |X_u^{t,x,(1)}|^2 \right]$ is $L^2$-integrable. By the same reasoning, it suffices to show that $\E[|X_u^{t,x,(2)}|^2]$ is uniformly bounded. To see this point, from the SDE 
\begin{align*}
    dX_u^{t,x,(2)}=&\left(b^{(1)}(X_u^{t,x},\alpha_u) X_u^{t,x,(2)}+b^{(2)}(X_u^{t,x},\alpha_u) |X_u^{t,x,(1)}|^2 \right) \,du \\
    &+ \left(\sigma^{(1)}(X_u^{t,x},\alpha_u) X_u^{t,x,(2)}+\sigma^{(2)}(X_u^{t,x},\alpha_u) |X_u^{t,x,(1)}|^2 \right)\, dV_u, \\
    X_t^{t,x,(2)}=& 0, 
\end{align*}
we can get an integral inequality for $ \E[|X_u^{t,x,(2)}|^{\beta}]$ as in \eqref{eq:SDEintegrability} where the constant $C$ depends on the moments of $X_u^{t,x,(1)}$ and the sup-norm of $b^{(2)},\sigma^{(2)}$. By Gr\"{o}nwall's inequality again, we get a uniform bound for $\E[|X_u^{t,x,(2)}|^{\beta}]$ and hence $x \mapsto G^{(2)}(x)$ is $L^2$-integrable. Repeating the argument, we prove that $G \in \mathbb{H}_{\lambda}$.

\vspace{4pt} 

\textit{Step II.} Similarly, under our assumptions it can be checked that $x \mapsto \E[f(X_s^{t,x,\alpha},\alpha_s)] \in \mathbb{H}_{\lambda}$ for all $s \in [t,T]$, and hence  $x \mapsto J(t,x,\alpha) \in \mathbb{H}_{\lambda}$, where
$$J(t,x,\alpha):=\E \left[\int_t^T f(X^{t,x,\alpha}_s, \alpha_s) \, ds + g(X^{t,x,\alpha}_T) \right].$$
Therefore we obtain the Lipschitz continuity in measure
$$J(t,\mu,\alpha)-J(t,\nu,\alpha)= \int J(t,x,\alpha) \, (\mu-\nu)(dx) \leq L \lVert \mu- \nu \rVert_{-\lambda}.$$
Regarding the regularity in time variable, due to Lemma~\ref{lem:measureflow_estimate}, it can be seen that 
\begin{align*}
    \left|J(t,\mu,\alpha)-J(s,\mu,\alpha) \right| & \leq \E \left[ \left|J(s,m_s^{t,\mu,\alpha},\alpha)-J(s,\mu,\alpha) \right| \right] \\
    & \leq L \E \left[ \lVert m_{s}^{t,\mu,\alpha}-\mu \rVert_{-\lambda} \right] 
     \leq L \sqrt{s-t}. 
\end{align*}

\end{proof}

To prove the next result, we apply It\^{o}-Wentzell formula and use the fact that the conditional law is the push forward measure along the common noise.

\begin{proof}[Proof of Proposition~\ref{lem:Ito}]
The proof of this result is divided into three steps. For simplicity of notation, we suppress $(t,\mu,\alpha)$ in notation $X^{t,\mu,\alpha}$. 

\vspace{4pt}
\textit{Step I:} Let us define two auxiliary processes
\begin{align*}
    Y_t= \int_0^t \tilde \sigma(\alpha_s) \, dW_s, \quad
    \tilde{X}_t=X_t - Y_t.
\end{align*}
Then $\tilde{X}_t$ satisfies the dynamics
\begin{align*}
    d\tilde{X}_t= b(\tilde{X}_t+Y_t, \alpha_t) dt + \sigma(\tilde{X}_t+Y_t, \alpha_t) \, dV_t. 
\end{align*}
Denote the conditional law of $\tilde{X}_t$ given $F^W_t$ by $\tilde{m}_t$. Then it can be seen that for almost every $\omega$, we have that 
\begin{align*}
    d\tilde m_t(\omega)(h)=\tilde m_t(\omega) \left(D h(\cdot)^\top b(\cdot+Y_t(\omega),\alpha_t(\omega))+\frac{1}{2} Tr (\mathcal{H}h(\cdot) \sigma \sigma^\top(\cdot + Y_t(\omega),\alpha_t(\omega)))\right) dt,
\end{align*}
and hence for almost every $\omega $, $t \mapsto \tilde m_t(\omega)$ the measure flow of some SDE. Also, we have $m_t(\omega)= (I_d+Y_t(\omega))_{\sharp} \tilde m_t$.

\vspace{4pt} 

\textit{Step II:} Define $\tilde m_s^y(\omega)= (id+y)_{\sharp}\tilde m_s$. Then we have that for each $y \in \R^d$
\begin{align}\label{eq:ito1}
    \psi(t,\tilde m^y_t(\omega))=&\psi(0,\tilde m^y_0(\omega))+\int_0^t \pa_t \psi(s,\tilde m^y_s(\omega)) \, ds \\
    &+ \int_0^t \,ds \int D_{\mu} \psi(s,\tilde m^y_s(\omega))(x) \cdot b(x+Y_t(\omega)-y,\alpha_t(\omega)) \, \tilde m^y_s(\omega)(dx)  \notag\\
    &+ \frac{1}{2} \int_0^t \, ds \int Tr(D_{x \mu} \psi(s,\tilde m^y_s(\omega))(x) \cdot \sigma\sigma^\top(x+Y_s(\omega)-y,\alpha_s(\omega))) \, \tilde m^y_s(\omega)(dx). \notag
\end{align}

\vspace{4pt} 
\textit{Step III:} For any $y \in \R^d$, we define a process $\Psi_t(y):= \psi(t,\tilde m^y_t)$. Note that $\psi(t,m_t)=\Psi_t(Y_t)$, and for each $\omega $, $y \mapsto \Psi_t(y)$ is second-order differentiable, and hence 
\begin{align*}
    D\Psi_t(Y_t)&= \lim\limits_{h \to 0} \frac{\psi(t,(id+h)_{\sharp}m_t)-\psi(t,m_t)}{h}=D_{\mu} \psi(t,m_t)[m_t], \\
    D^2 \Psi_t(Y_t)&= \mathcal{H}\psi(t,m_t).
\end{align*}
Now we invoke It\^{o}-Wentzell formula,
\begin{align*}
d\Psi_t(Y_t)= d\psi(t,m_t)+ D\Psi_t(Y_t) \cdot \tilde \sigma(\alpha_t) \, dW_t+ \frac{1}{2} Tr(D^2\Psi_t(Y_t) \cdot \tilde \sigma \tilde \sigma^\top (\alpha_t)) \, dt.
\end{align*}
Using \eqref{eq:ito1}, we conclude the result. 
\end{proof}

The proof of the viscosity is more or less standard. Most of proof is devoted to verify the assumption for uniqueness, which relies on the fact the Hamiltonian is defined as the infimum of integrals.

\begin{proof}[Proof of Theorem~\ref{thm:viscosity_property}]
\textit{Step I: Subsolution property} Suppose $\psi$ is partially $C^2$ regular such that $v -\psi$ obtains a local maximum at a point $(t,\mu) \in [0,T) \times \Pc_2(\R^d)$ and $v(t,\mu)=\psi(t,\mu)$. It suffices to show that for any fixed $a \in A$, 
\begin{align}\label{eq:fixeda}
    -\pa_t \psi(t,\mu)\leq   K(a,\mu, D_{\mu} \psi(t,\mu), D_{x \mu} \psi(t,\mu), \mathcal{H}\psi(t,\mu)). 
\end{align}

Suppose $\delta>0$ is a fixed constant such that $v(s,\nu) \leq \psi(s,\nu)$ for all $(s,\nu)$ satisfying $|s-t|+\rho_F(\nu,\mu) \leq \delta$. Take $\alpha \equiv a \in \mathcal{A}$, $\tau_{\delta}:=\inf \{s \geq t: \rho_F(m_s^{t,\mu,\alpha},\mu) \geq \delta \}$, and a sequence of stopping times $\tau_n:=\min\{ t+ 1/n, \tau_{\delta} \}$ for all $n \geq 0$. Then for any $n \geq 1/\delta$, thanks to Proposition~\ref{prop:Lipschizvalue}, we get 
\begin{align*}
    \psi(t,\mu) \leq \E \left[\int_t^{\tau_n} f(X_s^{t,\mu,\alpha},a) \, ds + \psi(\tau_n, m^{t,\mu,\alpha}_{\tau_n}) \right].
\end{align*}
For the simplicity of notation, write $m_s$ for $m^{t,\mu,\alpha}_s$. Note that $ \int_t^s D_{\mu} \psi(u,m_u)([m_u]) \, dW_u $ is a martingale for $s \leq \tau_n$ due to Lemma~\ref{lem:measureflow_estimate} and the fact that $D_{\mu} \psi(u,m_u)(\cdot) \in B_{\mathfrak{q}}^d$. Applying Lemma~\ref{lem:Ito} to the term $\psi(\tau_n, m_{\tau_n})$, we obtain 
\begin{align*}
0 \leq n \E \left[\int_t^{\tau_n} -\pa_t \psi(s,m_s) +K(a,m_s, D_{\mu} \psi(s,m_s), D_{x \mu} \psi(s,m_s),\Hc \psi(s,m_s) ) \,ds \right].
\end{align*}
Note that $n\tau_n \to 1 \ a.s.$ as $n \to \infty$ and $\rho(m_s,\mu) \to 0 \ a.s.$ as $s \to t$ thanks to the pathwise continuity of $m_s$. Then according to the continuity of $\psi, K$, the mean value theorem, and the dominated convergence theorem, we conclude \eqref{eq:fixeda}.

\vspace{4pt}

\textit{Step II: Supersolution Property.} Take $\psi \in C^2([0,T) \times \Pc_2(\R^d); \R)$ such that $v -\psi$ obtains a local minima at a point $(t,\mu) \in [0,T) \times \Pc_2(\R^d)$ and $v(t,\mu)=\psi(t,\mu)$. Towards a contradiction, suppose that for some positive $\epsilon >0$
\begin{align*}
    -\pa_t \psi(t,\mu) \leq  \inf_{a \in A}  K(a, \mu, D_{\mu} \psi(t,\mu), D_{x \mu} \psi(t,\mu), \mathcal{H}\psi(t,\mu))-2\epsilon. 
\end{align*}

 Due to the continuity in $\mu$ and the local minimal property of $(t,\mu)$, there exists some $\delta>0$ such that for all $(s,\nu)$ with $|s-t|+\rho_F(\mu,\nu)<\delta$ 
\begin{align}\label{eq:contradict}
    -\pa_t \psi(s,\nu) \leq  \inf_{a \in A}  K(a, \nu, D_{\mu} \psi(s,\nu), D_{x \mu} \psi(s,\nu), \mathcal{H}\psi(s,\nu))-\epsilon,
\end{align}
and 
\begin{align}\label{eq:localminima}
    v(s,\nu) \geq \psi(s,\nu). 
\end{align}
For any $n \in \mathbb{N}$ with $1/n <\delta$, thanks to Proposition~\ref{prop:Lipschizvalue}, we choose a control $\alpha^n$ and a stopping time $$\tau_n:= \min\left\{t+1/n, \, \inf\{s \geq t: \, \rho(\mu,m^{t,\mu,\alpha^n}_s) \geq \delta-1/n\} \right\} $$ 
satisfying 
\begin{align*}
    v(t,\mu) \geq \E \left[\int_t^{\tau_n} f(X_s^{t,\mu,\alpha^n}) \,ds + v(\tau^n, m_{\tau_n}^{t,\mu,\alpha^n}) \right]-\frac{1}{n^2}.
\end{align*}
Using \eqref{eq:localminima} and It\^{o}'s formula Lemma~\ref{lem:Ito}, we get that
\begin{align*}
    0 \geq & \E \left[-\psi(t,\mu)+\int_t^{\tau_n} f(X_s^{t,\mu,\alpha^n}) \,ds + \psi(\tau^n, m_{\tau_n}^n) \right]-\frac{1}{n^2} \\
    =&\E \left[\int_t^{\tau_n}\pa_t\psi(s,m^n_s)+K(\alpha^n_s,m^n_s, D_{\mu}\psi(s,m^n_s), D_{x \mu} \psi(s,\nu), \mathcal{H}\psi(s,m^n_s)) \,ds \right]-\frac{1}{n^2} ,
\end{align*}
where $m_s^n$ denotes $m_s^{t,\mu,\alpha^n}$. Now it follows from the inequality \eqref{eq:contradict} that 
\begin{align*}
    0 \leq \E\left[\int_t^{\tau_n} \epsilon \, ds \right]-\frac{1}{n^2},
\end{align*}
which contradicts with the fact that  $\lim_{n \to \infty} \E[n(\tau_n-t)]=1$.

\vspace{4pt}
\textit{Step III: Uniqueness.} 
    Given the Lipschitz continuity proven in Proposition \ref{prop:Lipschizvalue}, the uniqueness is a consequence of Theorem \ref{thm:comp} if we can prove that the generator $K$ satisfies the assumption of this theorem. To check the inequality \eqref{eq:unifh1}, we estimate
\begin{align*}
    &|\inf_{a\in A} K(a,\mu,f,g,X^{22})-\inf_{a\in A}K( a,\mu,\tilde f ,\tilde g,\tilde X^{22})|\\
    &\leq\sup_{a\in A} | K(a,\mu,f,g,X^{22})-K(a, \mu,\tilde f ,\tilde g,\tilde X^{22})|\\
    &\leq C\left( \int |f(x) -\tilde f(x)|+|g(x)-\tilde g(x)|\, \mu(dx)+  |X^{22}-\tilde X^{22}|\right)\\
    &\leq C\left( (||f(x) -\tilde f(x)||_{\mathfrak{q}}+||g(x)-\tilde g(x)||_{\mathfrak{q}})(1+Tr(V(\mu))+|m(\mu)|^2)+  |X^{22}-\tilde X^{22}|\right)
\end{align*}
for a constant that only depends on the supremum of $f,b,\sigma,\tilde \sigma$.

We now verify the assumption \eqref{eq:contpsi} for the Hamiltonian $K$ by estimating 
\begin{align*}
   & \inf_a K\left(a,\mu,\frac{m(\mu-\tilde \mu)}{\e}+\frac{D_\mu \Psi(\mu)}{2\e},\frac{D_{x\mu} \Psi(\mu)}{2\e},X^{22}\right)\\
   &\quad\quad-\inf_a K\left(a,\tilde \mu,\frac{m(\mu-\tilde \mu)}{\e}+\frac{D_\mu \Psi(\mu)}{2\e},\frac{D_{x\mu}\Psi(\mu)}{2\e},-\tilde X^{22}\right)\\
   & \leq \sup_a \Big\{ \Big|\int  f(x,a)(\mu-\tilde \mu)(dx)\Big|+\frac{1}{2} Tr( \tilde \sigma \tilde \sigma^\top(a) (X^{22}+\tilde X^{22}))\Big\}\\
   &+\sup_a\left\{\Big|\int b^\top(x,a) \frac{m(\mu-\tilde \mu)}{\e} (\mu-\tilde \mu)(dx)\Big|+\Big|\int b(x,a) \frac{D_\mu \Psi(\mu)}{2\e} (\mu-\tilde \mu)(dx)\Big|\right\}\\
   &+\frac{1}{2} \sup_a\left\{\Big|\int Tr\left( \sigma \sigma^\top (x,a) \frac{D_{x\mu}\Psi(\mu)}{2\e}\right)\, (\mu-\tilde \mu)(dx) \Big|\right\}
\end{align*}
for $\mu,\tilde \mu $ so that $\rho_F(\mu,\tilde \mu)\leq C_0\e$ for some $C_0>0$. 
Denote
\begin{align*}
    I_1&:=\sup_a \Big\{ \Big|\int  \left(f(x,a)+b^\top(x,a) \frac{m(\mu-\tilde \mu)}{\e}\right)(\mu-\tilde \mu)(dx)\Big|+\frac{1}{2} Tr( \tilde \sigma \tilde \sigma^\top(a) (X^{22}-\tilde X^{22}))\Big\}\\
    I_2&:=\sup_a\left\{\Big|\int b^\top(x,a) \frac{D_\mu \Psi(\mu)}{2\e} (\mu-\tilde \mu)(dx)\Big|\right\}\\
    I_3&:=\sup_a\left\{\Big|\int Tr\left( \sigma \sigma^\top (x,a) \frac{D_{x\mu}\Psi(\mu)}{2\e}\right)\, (\mu-\tilde \mu)(dx) \Big|\right\}
\end{align*}
Right hand side of \eqref{eq:condcomp} shows that $X+\tilde X\leq 0$ for the order of symmetric matrices. Thus, $X^{22}+\tilde X^{22}\leq 0$. By Lemma \ref{lem:metricrho}, we have
\begin{align*}
I_1&\leq  C\sup_{a\in A}\left\{||f(\cdot,a)||_{\lambda} \rho_F(\mu,\tilde\mu)+\frac{\rho_F(\mu,\tilde \mu)}{\e}\sup_j||b^j(\cdot,a)||_{\lambda} \rho_F(\mu,\tilde\mu)\right\}\\
&\leq  C\left(\rho_F(\mu,\tilde\mu)+\frac{\rho^2_F(\mu,\tilde \mu)}{\e}\right).
\end{align*}
To estimate $I_2$, for given $a\in A$, define the functions
\begin{align*}
    d_1(x,a)&:= 2b^\top(x,a)\Big(V( \Sc_0(\mu))-V(\Sc_0(\tilde \mu))\Big)(x-m( \mu))\\
    d_2(x,a)&:=b^\top(x,a)\left(\int \frac{Re(i k F_k(\Sc_0(\mu))(F_k( \Sc_0(\mu))-F_k(\Sc_0(\tilde \mu)))^*)}{(1+|k|^2)^\lambda}dk\right)\\
    d_3(x,a)&:=-b^\top(x,a)\int \frac{Re(i k f_k(x-m( \mu))(F_k( \Sc_0(\mu))-F_k(\Sc_0(\tilde \mu)))^*)}{(1+|k|^2)^\lambda}dk\\
    d_4(x)&:=-\int \frac{Re(i k f_k(x)(F_k( \Sc_0(\mu))-F_k(\Sc_0(\tilde \mu)))^*)}{(1+|k|^2)^\lambda}dk
\end{align*}
so that
\begin{align*}
    \int b^\top(x,a)\frac{D_\mu\Psi(\mu)}{2\e}(\mu-\tilde\mu)(dx)=\frac{1}{\e}\int (d_1(x,a)+d_2(x,a)+d_3(x,a))(\mu-\tilde\mu)(dx)
\end{align*}
We need to estimate $\lambda$ norms of $(d_1(\cdot,a),d_2(\cdot,a),d_3(\cdot,a))$.
To bound the first two terms we use the inequalities 
\begin{align*}
 |V(\Sc_0(\mu))-V(\Sc_0(\tilde{\mu}))| &\leq \rho_F(\mu,\tilde{\mu}) \\
 \int \frac{Re(i k F_k(\Sc_0(\mu))(F_k( \Sc_0(\mu))-F_k(\Sc_0(\tilde \mu))))}{(1+|k|^2)^\lambda}dk &\leq \rho_F(\mu,\tilde{\mu}) \sqrt{\int \frac{| k |^2}{(1+|k|^2)^\lambda}dk} 
\end{align*}
to get that  
\begin{align*}
     &\frac{1}{\e}\Big|  \int (d_1(x,a)+d_2(x,a))(\mu-\tilde\mu)(dx)\big|\leq \frac{\rho_F(\mu,\tilde \mu)}{\e}\sup_{a\in A}\{||d_1(\cdot,a)||_\lambda+||d_1(\cdot,a)||_\lambda\}\\
     &\leq \frac{(1+|m(\mu)|)\rho^2_F(\mu,\tilde \mu)}{\e} \sup_{i,j,a}\{||x^i b^j(x,a)||_\lambda+||b^j(x,a)||_\lambda \}\left(1+\sqrt{\int \frac{| k |^2}{(1+|k|^2)^\lambda}dk}\right)\\
     &\leq \frac{C(1+|m(\mu)|) \rho_F^2(\mu,\tilde{\mu})}{\e}.
\end{align*}
Regarding the last term, by Lemma \ref{lem:metricrho} we also have
$$\frac{1}{\e}\Big|\int d_3(x,a)(\mu-\tilde\mu)(dx)\Big|\leq \frac{C}{\e}||d_3(\cdot,a)||_\lambda \rho_F(\mu,\tilde \mu)$$
and we need to estimate $||d_3(\cdot,a)||_\lambda.$
Since $\lambda>3$, the convexity of $x\mapsto (1+|x|^2)^{\lambda/2}$ leads to 
$$2^{-\lambda}(1+|k|^2)^{\lambda/2}\leq \left(1+\Big|\frac{k-l+l}{2}\Big|^2\right)^{\lambda/2}\leq \frac{1}{2}\left((1+|k-l|^2)^{\lambda/2}+(1+|l|^2)^{\lambda/2}\right)$$
for all $k,l\in\R^d.$
Using the convolution theorem we have
\begin{align*}
    &||d_3(\cdot,a)||^2_\lambda =\int \Big| \int(1+|k|^2)^{\lambda/2}\Fc b(k-l,a)e^{-i l^\top m(\mu)}\Fc d_4(l) \,dl\Big|^2dk\\    
    &\leq2^{\lambda-1}\int \Big| \int\left((1+|k-l|^2)^{\lambda/2}+(1+|l|^2)^{\lambda/2}\right)|\Fc b(k-l,a)||\Fc d_4(l)|\,dl\Big|^2dk\\    
    &\leq 2^{\lambda}\int \Big| \int(1+|k-l|^2)^{\lambda/2}|\Fc b(k-l,a)||\Fc d_4(l)|dl\Big|^2dk\\        
    &+2^{\lambda}\int \Big| \int(1+|l|^2)^{\lambda/2}|\Fc b(k-l,a)||\Fc d_4(l)|dl\Big|^2dk.  
    \end{align*}
    Thanks to Fourier inversion theorem, we have the explicit formula of $\mathcal{F}d_4(l)$,
$$|\mathcal{F}d_4(l)|=\frac{ |l| |F_l( \Sc_0(\mu))-F_l(\Sc_0(\tilde \mu))|}{(1+|l|^2)^{\lambda}}.$$ Thus we obtain that 
    \begin{align*}
    ||d_3(\cdot,a)||^2_\lambda   &\leq 2^{\lambda}\int  \int\frac{(1+|k-l|^2)^{\lambda}|\Fc b(k-l,a)|^2}{(1+|l|^2)^{\lambda-2}}dl dk\int \frac{ |l|^2 |F_l( \Sc_0(\mu))-F_l(\Sc_0(\tilde \mu))|^2}{(1+|l|^2)^{2\lambda-\lambda+2}}dl\\        
    &+2^{\lambda}\int \Big| \int|\Fc b(k-l,a)| \frac{ |l| |F_l( \Sc_0(\mu))-F_l(\Sc_0(\tilde \mu))|}{(1+|l|^2)^{\lambda/2}}dl\Big|^2dk \\
    &\leq 2^{\lambda}||b(\cdot,a)||^2_\lambda\rho_F^2(\mu,\tilde \mu)\int  \frac{1}{(1+|l|^2)^{\lambda-2}}dl \\        
    &+2^{\lambda}\int \Big| \int\frac{|\Fc b(k-l,a)|}{(1+|l|^2)^{\frac{d}{4}+\frac{1}{2}}}\frac{ |F_l( \Sc_0(\mu))-F_l(\Sc_0(\tilde \mu))|}{(1+|l|^2)^{\frac{\lambda}{10+d}}} \frac{ 1}{(1+|l|^2)^{\frac{\lambda}{2}-\frac{1}{2}-\frac{d}{4}-\frac{1}{2}-\frac{\lambda}{10+d}}}dl\Big|^2dk.  
\end{align*}
Denote 
\begin{align}\label{eq:defcd}
C_d=\left(\int \frac{  1}{(1+|l|^2)^{\left(\frac{\lambda}{2}-\frac{1}{2}-\frac{d}{4}-\frac{1}{2}-\frac{\lambda}{10+d}\right)\frac{20+2d}{8+d}}}dl\right)^{\frac{8+d}{10+d}}.
\end{align}
Given the value of $\lambda=d+7$, the exponent is 
\begin{align*}
    &\left(\frac{\lambda}{2}-\frac{1}{2}-\frac{d}{4}-\frac{1}{2}-\frac{\lambda}{10+d}\right)\frac{20+2d}{8+d}\\
    &=d+7-\frac{d+4}{4}\frac{20+2d}{8+d}=d-\frac{d+4}{d+8}\frac{d}{2}+7-5\frac{d+4}{8+d}>\frac{d}{2}+2
\end{align*}
and $C_d<\infty$. 
We now apply H\"{o}lder inequality to obtain that
\begin{align*}
 &\Big| \int\frac{|\Fc b(k-l,a)|}{(1+|l|^2)^{\frac{d}{4}+\frac{1}{2}}}\frac{ |F_l( \Sc_0(\mu))-F_l(\Sc_0(\tilde \mu))|}{(1+|l|^2)^{\frac{\lambda}{10+d}}} \frac{ 1}{(1+|l|^2)^{\frac{\lambda}{2}-\frac{1}{2}-\frac{d}{4}-\frac{1}{2}-\frac{\lambda}{10+d}}}dl\Big| \\
 & \leq \left|\int\frac{|\Fc b(k-l,a)|^2}{(1+|l|^2)^{\frac{d}{2}+1}} dl \right|^{1/2} \left| \int     \frac{ |F_l( \Sc_0(\mu))-F_l(\Sc_0(\tilde \mu))|^{10+d}}{(1+|l|^2)^{\lambda}}  dl \right|^{\frac{1}{10+d}} \\ 
 & \quad \ \ \times \left| \int  \frac{1}{(1+|l|^2)^{\left(\frac{\lambda}{2}-\frac{1}{2}-\frac{d}{4}-\frac{1}{2}-\frac{\lambda}{10+d}\right)\frac{20+2d}{8+d}} }dl \right|^{\frac{8+d}{20+2d}},
\end{align*}
and hence
\begin{align*}
    ||d_3(\cdot,a)||^2_\lambda &\leq C||b(\cdot,a)||^2_\lambda\rho_F^2(\mu,\tilde \mu)\\        
    &+C_d\int  \int\frac{|\Fc b(k-l,a)|^2}{(1+|l|^2)^{\frac{d}{2}+1}}dldk\left(\int\frac{ |F_l( \Sc_0(\mu))-F_l(\Sc_0(\tilde \mu))|^{10+d}}{(1+|l|^2)^\lambda}dl\right)^{\frac{2}{10+d}}. 
\end{align*}
Given the bound $|F_l( \Sc_0(\mu))-F_l(\Sc_0(\tilde \mu))|<2$, we have that 
\begin{align*}
    ||d_3(\cdot,a)||^2_\lambda &\leq C\left(||b(\cdot,a)||^2_\lambda\rho_F^2(\mu,\tilde \mu)+C_d||b(\cdot,a)||^2_{L_2}\rho_F^{\frac{2}{10+d}}(\mu,\tilde \mu)\right).
\end{align*}
Given our assumption on $b$, this implies
\begin{align*}
    \sup_{a\in A}||d_3(\cdot,a)||_\lambda &\leq C\rho_F^{\frac{1}{10+d}}(\mu,\tilde \mu)\sqrt{1+\rho_F^{\frac{18+2d}{10+d}}(\mu,\tilde \mu)}
\end{align*}
and 
\begin{align*}
I_2&\leq C\frac{\rho_F^{\frac{11+d}{10+d}}(\mu,\tilde \mu)}{\epsilon}\sqrt{1+\rho_F^{\frac{18+2d}{10+d}}(\mu,\tilde \mu)}+\frac{C(1+|m(\mu)|) \rho_F^2(\mu,\tilde{\mu})}{\e}\\
&\leq CC_0\rho_F^{\frac{1}{10+d}}(\mu,\tilde \mu)\sqrt{1+\rho_F^{\frac{18+2d}{10+d}}(\mu,\tilde \mu)}+CC_0(1+|m(\mu)|) \rho_F(\mu,\tilde{\mu})
\end{align*}

To estimate $I_3$ we proceed similarly to the estimation of $I_2$. However, in this case, the term $d_4$ is now
$$-\int \frac{Re(i kk^\top f_k(x)(F_k( \Sc_0(\mu))-F_k(\Sc_0(\tilde \mu)))^*)}{(1+|k|^2)^\lambda}dk.$$
Thus, as in the estimate of $||d_3(\cdot,a)||^2_\lambda$, we have 
$$|\Fc d_4(l)|=\frac{ |l|^2 |F_l( \Sc_0(\mu))-F_l(\Sc_0(\tilde \mu))|}{(1+|l|^2)^{\lambda}}$$ instead of 
$$\frac{ |l| |F_l( \Sc_0(\mu))-F_l(\Sc_0(\tilde \mu))|}{(1+|l|^2)^{\lambda}}$$
and $C_d$ in \eqref{eq:defcd} has to be defined with an exponent 
\begin{align}\label{eq:lambdagood}
\left(\frac{\lambda}{2}-\frac{1}{2}-\frac{d}{4}-1-\frac{\lambda}{10+d}\right)\frac{20+2d}{8+d}.
\end{align}
This exponent is still strictly larger than $\frac{d}{2}$, $C_d<\infty$, and 
\begin{align*}
I_3&\leq C\rho_F^{\frac{1}{10+d}}(\mu,\tilde \mu)\sqrt{1+\rho_F^{\frac{18+2d}{10+d}}(\mu,\tilde \mu)}+\frac{C(1+|m(\mu)|) \rho_F^2(\mu,\tilde{\mu})}{\e}\\
&\leq C\rho_F^{\frac{1}{10+d}}(\mu,\tilde \mu)\sqrt{1+\rho_F^{\frac{18+2d}{10+d}}(\mu,\tilde \mu)}+CC_0(1+|m(\mu)|) \rho_F(\mu,\tilde{\mu})
\end{align*}
which is \eqref{eq:contpsi}.

\end{proof}

\bibliographystyle{siam}
\bibliography{ref.bib}

\end{document}